\documentclass[11pt]{amsart}

\usepackage{amssymb,mathtools}
\usepackage{microtype}
\usepackage[letterpaper,textwidth=6.2in,textheight=8.6in,
  headheight=14pt,headsep=18pt,centering]{geometry}
\usepackage[hidelinks]{hyperref}
\usepackage{bookmark}
\usepackage{needspace}
\hypersetup{
  pdftitle={P-adic Curved Kakeya Sets, Projection Theorem and Covering Numbers},
  pdfsubject={Curved Kakeya sets with polynomial parametrizations, evaluation dimensions, and anisotropic coverings},
  pdfkeywords={p-adic curved Kakeya sets, restricted directions,
    Hausdorff dimension, polynomial projections, Kakeya maximal inequalities,
    exceptional parameters, polynomial-to-line lifting,
    anisotropic covering numbers}
}
\allowdisplaybreaks[1]
\issueinfo{}{0}{}{}
\PII{\mbox{}\hfill}
\copyrightinfo{}{}

\AddToHook{cmd/section/before}{\Needspace{7\baselineskip}}
\AddToHook{env/theorem/before}{\Needspace{6\baselineskip}}
\AddToHook{env/corollary/before}{\Needspace{6\baselineskip}}
\newtheorem{theorem}{Theorem}[section]
\newtheorem{proposition}[theorem]{Proposition}
\newtheorem{lemma}[theorem]{Lemma}
\newtheorem{corollary}[theorem]{Corollary}
\theoremstyle{definition}
\newtheorem{definition}[theorem]{Definition}
\theoremstyle{remark}
\newtheorem{remark}[theorem]{Remark}
\newcommand{\Q}{\mathbb Q}
\newcommand{\Z}{\mathbb Z}
\newcommand{\PP}{\mathbb P}
\newcommand{\HH}{\mathcal H}
\newcommand{\dimH}{\dim_{\mathrm H}}
\newcommand{\red}{\operatorname{red}}
\newcommand{\one}{\mathbf 1}
\newcommand{\defeq}{\mathrel{:=}}
\DeclarePairedDelimiter{\norm}{\lVert}{\rVert}

\begin{document}

\title[Curved Kakeya sets, projection theorem and covering numbers]
  {P-adic Curved Kakeya Sets,\\
   Projection Theorem and Covering Numbers}

\author{Yi Lou}
\address{DEPARTMENT OF MATHEMATICS, NATIONAL UNIVERSITY OF SINGAPORE, SINGAPORE}
\email{lou\_yi@u.nus.edu}

\date{}

\begin{abstract}
Let $n,d\geq1$, and suppose that $E\subset\Q_p^n$ contains a
degree-$d$ polynomial image of $\Z_p$ in every leading direction of a
nonempty set $V\subset\PP^{n-1}(\Q_p)$. We prove the sharp bound
$\dimH E\geq\dimH V+1$. We also show that for joint block-valued polynomial evaluations,
the total incidence dimension is at least one plus the supremum of
the image dimensions. For analytic coefficient sets, Haar-almost
every evaluation attains this supremum, and we bound the Hausdorff
dimension of the parameters where the image dimension drops. For
bounded coefficient sets and arbitrary fixed anisotropic block
scales, we prove a quantitative covering comparison with a common
Haar-null exceptional set. Outside this set, each evaluation's
covering number dominates the maximum over any compact reference
ball, up to any positive power loss at all sufficiently small scales.
The proofs use the polynomial-to-line lifting map due to Nadjimzadah,
together with Dhar's prime-power Kakeya set estimate and a tube maximal
inequality derived from it.
ChatGPT (OpenAI) was used to assist with mathematical discussions,
proofreading, and bibliographic searches.
\end{abstract}

\maketitle
\section{Introduction and main results}
\label{sec:introduction}

In this article, we study curved Kakeya sets with polynomial
parametrizations and the images of polynomial evaluation maps over
$\Q_p$. The proofs use the explicit \emph{polynomial-to-line lifting} map
from Nadjimzadah's polynomial construction \cite{Nadjimzadah}.
We first give a sharp dimension bound for a set containing polynomial
curves with prescribed leading directions. We then use the same
construction to compare the dimensions of evaluation images and to
estimate the exceptional parameters. Finally, we prove a comparison
of anisotropic covering numbers for bounded coefficient sets.

In real-variable harmonic analysis, curved Kakeya problems arise
naturally from norm estimates for variable-coefficient oscillatory
integral operators of H\"ormander type
\cite{Hormander,Wisewell}. The restriction to a
prescribed set of directions is motivated by the finite-field Kakeya
maximal and algebraic-curve estimates of Ellenberg, Oberlin, and Tao
\cite[Theorems~1.3 and~1.6]{EOT}. In particular, their work implies the
restricted curved Kakeya bound for polynomial images over finite
fields \cite[Corollary~1.10]{EOT}.

The quantitative finite-ring input is Dhar's prime-power Kakeya
set estimate \cite[Theorem~1.8]{DharSet}. It yields the
restricted-direction line theorem and the finite-ring covering
comparison. We also derive from this set estimate a tube maximal
inequality with a logarithmic loss; see
Lemma~\ref{lem:projection-tube-maximal}. Together with Frostman
measures and measurable selection, this inequality gives the
projection--incidence theorem. We then apply the lifting map to
obtain the corresponding results for polynomial families.

We fix a prime $p$ and normalize $|p|_p=p^{-1}$. Cartesian powers
of $\Q_p$ carry the supremum metric, and $\dimH$ denotes Hausdorff
dimension. The projective metric and Haar measure are specified in
Section~\ref{sec:lines}. Other notations are introduced when they are used.

\subsection*{Curved Kakeya sets and evaluation dimensions}
We consider curved Kakeya sets with polynomial parametrizations.
A degree-$d$ polynomial map with nonzero leading coefficient $v$
has \emph{leading direction} $[v]\in\PP^{n-1}(\Q_p)$.
A set containing the image of $\Z_p$ under such a map for every
direction in $V$ is called a \emph{degree-$d$ curved Kakeya set in
the directions $V$}. Our first result gives a lower bound for its
Hausdorff dimension in terms of $V$. No condition is imposed on the
lower coefficients.

\begin{theorem}[Curved Kakeya sets over $\Q_p$]
\label{thm:padic}
Let $n,d\geq1$, and let $V\subset\PP^{n-1}(\Q_p)$ be nonempty.
Let $E\subset\Q_p^n$. Suppose that for every $\omega\in V$ there is a
polynomial map
\begin{equation*}
 \gamma_\omega:\Q_p\longrightarrow\Q_p^n,\qquad
 \gamma_\omega(t)\defeq \sum_{j=0}^{d-1}a_{\omega,j}t^j+v_\omega t^d,
\end{equation*}
where $a_{\omega,j}\in\Q_p^n$, $v_\omega\neq0$, and
$[v_\omega]=\omega$, such that
\begin{equation*}
 \gamma_\omega(\Z_p)\subset E.
\end{equation*}
Then
\begin{equation*}
 \dimH E\geq\dimH V+1.
\end{equation*}
In particular, if $V=\PP^{n-1}(\Q_p)$, then $\dimH E=n$.
\end{theorem}

When $d=1$, this is a statement about line segments. For higher
degrees, we apply the line-segment result to the inverse image under
the lifting map. We shall compute the dimensions of this inverse
image and its direction set in Section~\ref{sec:lifting}.

We next consider polynomials whose coefficients belong to a fixed
set. The aim is to compare their evaluation images at different
parameters.

To define these maps, fix integers $n,d\geq1$ and scalar weights
$a_0,\ldots,a_d\in\Q_p$ with $a_d\neq0$. Define
\begin{equation}\label{eq:projection-family}
 P_r:(\Q_p^n)^{d+1}\longrightarrow\Q_p^n,\qquad
 P_r(v_0,\ldots,v_d)\defeq\sum_{j=0}^d a_jr^jv_j,
 \qquad r\in\Q_p.
\end{equation}
For fixed $r$, the map $P_r$ is linear in the $d+1$ coefficient
vectors. We call these evaluation maps \emph{polynomial projections}.
The fixed weights satisfy only $a_d\neq0$; the lower weights may
vanish. We prove the following comparison with $P_0(\mathcal E)$.

\begin{theorem}[Polynomial projections and exceptional parameters]
\label{thm:polynomial-projections}
For the family \eqref{eq:projection-family}, let
$\mathcal E\subset(\Q_p^n)^{d+1}$ be nonempty and analytic, and set
$S\defeq \dimH P_0(\mathcal E)$. Then
\begin{equation}\label{eq:projection-ae}
 \dimH P_r(\mathcal E)\geq S
 \quad\text{for Haar-almost every }r\in\Q_p.
\end{equation}
For every $0\leq\sigma<S$,
\begin{equation}\label{eq:projection-exceptional}
 \dimH\{r\in\Q_p:\dimH P_r(\mathcal E)\leq\sigma\}
 \leq 1-\frac{S-\sigma}{dn}.
\end{equation}
\end{theorem}

Here a subset of a Polish space is called \emph{analytic} if it
is a continuous image of a Polish space. Every Borel set is analytic
\cite{Kechris}. Analyticity is assumed for the individual-image
dimension results, where we choose Frostman measures and measurable
families of lines. The lifting construction and the total-incidence
bounds in Propositions~\ref{prop:polynomial-total-incidence}
and~\ref{prop:block-total-incidence} apply to arbitrary nonempty
coefficient sets.

By translating the parameter, we may use any point of $\Q_p$ as the reference. This is proved in
Proposition~\ref{prop:arbitrary-reference-parameter}. Choosing a
sequence of reference images whose dimensions approach the
supremum then gives the following corollary.

\begin{corollary}[Almost-everywhere maximal projection dimension]
\label{cor:maximal-projection-dimension}
Under the hypotheses of Theorem~\ref{thm:polynomial-projections}, put
\begin{equation*}
 S_*\defeq \sup_{s\in\Q_p}\dimH P_s(\mathcal E).
\end{equation*}
Then
\begin{equation}\label{eq:projection-maximal-ae}
 \dimH P_r(\mathcal E)=S_*
 \quad\text{for Haar-almost every }r\in\Q_p.
\end{equation}
For every $0\leq\sigma<S_*$,
\begin{equation}\label{eq:projection-maximal-exceptional}
 \dimH\{r\in\Q_p:\dimH P_r(\mathcal E)\leq\sigma\}
 \leq 1-\frac{S_*-\sigma}{dn}.
\end{equation}
In particular, a lower bound $\dimH P_{r_0}(\mathcal E)\geq S$ at any
one parameter implies $\dimH P_r(\mathcal E)\geq S$ for Haar-almost
every $r$.
\end{corollary}

We note that the corollary compares the dimensions of the
evaluation images, rather than comparing the dimensions to
$\dimH\mathcal E$ alone. In the scalar case
$n=1$, Shen \cite[Theorem~1.4]{Shen} gives a stronger
exceptional-parameter estimate for the scalar linear family $x+ry$. Still in the scalar case $n=1$, $a_j\neq0$ and $p>d+1$, Johnsrude and Lin
\cite[Theorem~1.1]{JL} have shown
$\dimH P_r(\mathcal E)=\min\{1,\dimH\mathcal E\}$ for Borel
$\mathcal E$ and almost every $r\in\Z_p$, after a diagonal change of coefficients. In the vector-valued setting, all spatial coordinates use the same
parameter, and the dimension of $\mathcal E$ need not be preserved. We give an example in
Remark~\ref{rem:coefficient-dimension-counterexample} and discuss
the sharpness of the bounds in Remark~\ref{rem:projection-sharpness}. In addition, for the block evaluation introduced below, Johnsrude--Lin's formulation in \cite{JL} does not apply to general coefficient vectors and scalar weights.

\subsection*{Block evaluations and anisotropic covering numbers}
We now extend the evaluation family to several output blocks.
The blocks may have different dimensions and degrees, but the
parameter is the same in every block. Fix integers $k\geq1$,
$n_i\geq1$, and $d_i\geq0$. For each $i$, let
$a_{i,j}\in\Q_p$, $0\leq j\leq d_i$, satisfy
$a_{i,d_i}\neq0$. Set
\begin{equation*}
 \mathcal V\defeq\prod_{i=1}^k(\Q_p^{n_i})^{d_i+1},
 \qquad N\defeq\sum_{i=1}^k n_i,
\end{equation*}
and, for $v=(v_{i,j})\in\mathcal V$, define
\begin{equation}\label{eq:block-evaluation-family}
\begin{aligned}
 P^i_r(v_{i,0},\ldots,v_{i,d_i})
       &\defeq\sum_{j=0}^{d_i}a_{i,j}r^jv_{i,j},\\
 \widetilde P_r(v)
       &\defeq\bigl(P^i_r(v_{i,0},\ldots,v_{i,d_i})\bigr)_{i=1}^k
       \in\Q_p^N.
\end{aligned}
\end{equation}
In particular, $\widetilde P_0(v)=(a_{i,0}v_{i,0})_{i=1}^k$.
Each block has $d_i+1$ coefficient vector, namely $v_{i,0},\cdots,v_{i,d_i}$. As in the single-output
family, only the leading scalar weights are assumed to be nonzero.
A block of degree $d_i=0$ is constant in $r$.

The exceptional-set bounds depend on the quantity
\begin{equation}\label{eq:block-denominator}
 \Lambda\defeq\sum_{i=1}^k n_i\max\{1,d_i\}.
\end{equation}
If every $d_i\geq1$, this is just $\sum_i n_i d_i$. The auxiliary
notation for the block lifts will be introduced in
Section~\ref{sec:block-evaluations}.

For any nonempty $\mathcal E\subset\mathcal V$, put
$S_*\defeq\sup_{s\in\Q_p}\dimH\widetilde P_s(\mathcal E)$.
Proposition~\ref{prop:block-total-incidence} gives the block version of
Proposition~\ref{prop:polynomial-total-incidence}:
\begin{equation*}
 \begin{aligned}
 S_*+1
 &\leq\dimH\{(r,\widetilde P_r(v)):r\in\Q_p,\ v\in\mathcal E\}\\
 &\leq\min\{\dimH\mathcal E+1,N+1\}.
 \end{aligned}
\end{equation*}
These inequalities do not require any regularity assumptions. Like the single-block case, for individual images, we need to assume that
$\mathcal E$ is analytic. In this case,
Theorem~\ref{thm:block-valued-projections} and
Corollary~\ref{cor:block-reference-maximal} give
\begin{gather*}
 \dimH\widetilde P_r(\mathcal E)=S_*
       \quad\text{for Haar-almost every }r,\\
\dimH\{r:\dimH\widetilde P_r(\mathcal E)\leq\sigma\}
       \leq1-\frac{S_*-\sigma}{\Lambda}
       \qquad(0\leq\sigma<S_*).
\end{gather*}

We also compare covering numbers of the evaluation images at different
scales in the output blocks. This choice is motivated by dynamical
dimension arguments, where diagonal actions expand different weight
spaces at different rates. If block $i$ expands by $e^{\alpha_i t}$,
the inverse image of a unit product ball has block radii
\begin{equation*}
 e^{-\alpha_i t}=\delta^{\alpha_i},\qquad \delta=e^{-t}.
\end{equation*}
Thus covering a set by these rectangles is equivalent to covering its
image under the diagonal action by unit product balls. Polynomial
evaluations arise naturally in this setting. In an irreducible
$\mathrm{SL}_2$-representation, applying an upper-triangular unipotent
element $u_r$ makes the highest-weight coordinate polynomial in $r$.
For the normalization
$a_t=\operatorname{diag}(e^{t/2},e^{-t/2})$, this coordinate expands
by $e^{d_i t/2}$ when the representation has dimension $d_i+1$,
giving the covering exponent $\alpha_i=d_i/2$. This description
is used explicitly by Ohm and Lin \cite[Sections~1 and~3]{OhmLin}.
The role of such rectangular geometry in dimension arguments is
illustrated by Lindenstrauss, Mohammadi, Wang and Yang
\cite[Lemma~6.2 and Sections~6--8]{LMWY}, who use boxes aligned with
the diagonal weight spaces to control concentration under
$\operatorname{Ad}(a_tu_r)$. Their restricted projection input comes
from Gan, Guo and Wang \cite[Theorem~2.1]{GGW}, and the resulting
localized energy estimates enter an iterative dimension argument
through modified Margulis functions. B\'enard and He
\cite[Sections~2.1 and~4.3]{BenardHe} use rectangular covering
numbers in a related multislicing argument. In suitable charts,
inverse images of small balls have scales $1,\delta^{1/2},\delta$,
and mass estimates for these rectangles yield dimension increments
for random walks. A corresponding projection-and-expansion argument
appears over $\Q_p$ in Lin
\cite[Section~4, proof of Proposition~4.1]{LinDensity}, for $p>3$:
discretized dimension is first transferred to a polynomially varying
expanding coordinate, and a diagonal element then expands the
resulting set to unit size. These arguments motivate the anisotropic
covering numbers considered here.

For $\boldsymbol\alpha\in[0,\infty)^k$ and $0<\delta<1$, let
$N_\delta^{\boldsymbol\alpha}(A)$ denote the least number of product
balls with block radii $\delta^{\alpha_i}$ needed to cover $A$.
The precise lattice definition is given in
Definition~\ref{def:block-rectangles}. We prove the following
comparison for bounded coefficient sets.

\begin{theorem}[Almost-everywhere anisotropic maximality]
\label{thm:intro-anisotropic-maximality}
For the block family \eqref{eq:block-evaluation-family}, let
$\mathcal E\subset\mathcal V$ be nonempty and bounded, and fix
$\boldsymbol\alpha\in[0,\infty)^k$. There is a Haar-full set
$G\subset\Q_p$ depending only on $\mathcal{E}$ and $\boldsymbol \alpha$, such that, for every $r\in G$, every compact ball
$B\subset\Q_p$ of positive radius, and every $\varepsilon>0$,
\begin{equation*}
 N_\delta^{\boldsymbol\alpha}(\widetilde P_r(\mathcal E))
 \geq\delta^\varepsilon
       \max_{s\in B}N_\delta^{\boldsymbol\alpha}
                           (\widetilde P_s(\mathcal E))
 \qquad(0<\delta<\delta_0),
\end{equation*}
where $\delta_0=\delta_0(r,B,\varepsilon)>0$. The parameter $r$ need not lie in $B$ and the scale threshold is uniform when the
references lie in a fixed compact ball.
\end{theorem}

We deduce this theorem from a finite-scale estimate, rather than
from the preceding Hausdorff dimension statements. More precisely,
let $\mu_B$ denote normalized Haar measure on $B$, let
$\delta=p^{-\ell}$, and put
$M_{\delta,B}=\max_{s\in B}N_\delta^{\boldsymbol\alpha}
(\widetilde P_s(\mathcal E))$. Theorem~\ref{thm:block-anisotropic-covering}
states that
\begin{equation*}
 \mu_B\{r\in B:
 N_\delta^{\boldsymbol\alpha}(\widetilde P_r(\mathcal E))
                   <\eta M_{\delta,B}\}
 \leq C(1+\ell)^{(\Lambda+1)/\Lambda}\eta^{1/\Lambda},
 \qquad 0<\eta<1.
\end{equation*}
For fixed family data, $B$, and $\boldsymbol\alpha$, one constant
$C$ works for every coefficient set in a fixed bounded box.
More generally, only the actual coefficients $a_{i,j}v_{i,j}$ need
be bounded. No measurability assumption is required.
In particular, a lower bound at one reference parameter gives
an almost-everywhere lower bound at the same block scales.
For example, suppose that
$N_\delta^{\boldsymbol\alpha}(\widetilde P_{r_0}(\mathcal E))
\geq\delta^{-s}$ for all sufficiently small $\delta$. The theorem
then gives the lower bound $\delta^{-s+\varepsilon}$ at almost
every parameter. When $\delta=e^{-t}$, the loss is
$\delta^\varepsilon=e^{-\varepsilon t}$. Since $\varepsilon$ is
arbitrary, the loss in the exponential covering-growth rate can be
made arbitrarily small.

Related anisotropic covering estimates were established by
B\'enard and He \cite[Theorem~2.1 and Section~4.3]{BenardHe}
in their work on multislicing and effective equidistribution.
Their multislicing theorem gives lower bounds for coverings by
nonlinear rectangles with unequal side lengths under
non-concentration hypotheses on the underlying set and the
associated flag distribution. These estimates yield dimension
increments for random walks on homogeneous spaces. Our result
addresses a complementary question over $\Q_p$: for arbitrary
fixed block scales and bounded joint coefficient sets, almost
every polynomial evaluation preserves, up to a factor
$\delta^\epsilon$, the largest reference-image covering number
on each compact parameter ball. Thus our conclusion
concerns preservation of existing covering complexity, rather
than the strict dimension improvement obtained through multislicing.

\subsection*{Proof strategy and organization}
We briefly describe the proofs. The algebraic construction is
Nadjimzadah's lifting map
\cite[Proposition~2.5(ii), proof in Section~4]{Nadjimzadah}.
After omitting the parameter from the output, the map is
\begin{equation*}
 \Phi(x_0,\ldots,x_{d-1},z)
       =\sum_{j=0}^{d-1}z^jx_j.
\end{equation*}
A degree-$d$ polynomial with prescribed leading vector is the image
under $\Phi$ of an affine line with $d-1$ freely chosen slope blocks.
A triangular polynomial change of coordinates also identifies
$\Phi^{-1}(E)$ with $E\times\Q_p^{(d-1)n+1}$. Nadjimzadah uses the
same telescoping construction over $\mathbb R$ for smooth
H\"ormander-type families. We use its algebraic identity over $\Q_p$
for arbitrary leading-direction sets and joint coefficient sets.

To prove the curved Kakeya theorem, we first use unit
reparametrizations to obtain a set of leading vectors with dimension
one greater than that of the prescribed directions. We apply the
restricted-direction line theorem to the lift and use the product
identity to remove the auxiliary coordinates.

For total incidence, we reverse the evaluation polynomial so that
the reference value becomes its leading coefficient. Adjoining a
scalar monomial allows us to apply the curved Kakeya theorem.
For several output blocks, we first group the coefficients of equal
powers. Translation then gives the result for arbitrary references.

To estimate individual images, we use the fact that evaluations
are slices of an incidence set. A related description appears in
K\"aenm\"aki--Orponen--Venieri \cite[Section~1, equation~(1.3)]{KOV}, who identify
projections in a non-degenerate circular family with vertical sections
of a union of curves and control concentration through tangency
incidences. In our proof, we lift the curves to lines and apply the
projection--incidence theorem. The slope set and each slice at a
nonzero parameter contain the same full auxiliary factor. We
compute its dimension and subtract it from the slice bounds.
For the block family, this is done separately in each block.

Finally, for covering numbers we carry out the same fiber and
direction count over a finite ring. We multiply the output blocks
by suitable powers of $p$ so that all covering scales correspond
to one residue modulus. These multiplications are contractions
after a fixed initial normalization. The resulting estimate is
uniform in the coefficient set, so a maximizing reference may be
chosen at each scale. Applying Borel--Cantelli then gives one
exceptional null set.

The article is organized as follows. Section~\ref{sec:lines} proves
the restricted-direction line theorem and the product dimension
identity. Section~\ref{sec:lifting} constructs the lift and proves
Theorem~\ref{thm:padic}, followed by sharpness and variants over finite
extensions. Section~\ref{sec:measure-theoretic-input}
then develops the measure-theoretic tools, derives the tube maximal
inequality from Dhar's set estimate, and proves the
projection--incidence theorem.
Sections~\ref{sec:polynomial-projections} and~\ref{sec:block-evaluations}
apply these results first to a single output block and then to joint
block families, including total incidence and arbitrary-reference
comparisons. Finally, Section~\ref{sec:anisotropic-covering} gives the
expansion interpretation, proves the uniform finite-ring and
anisotropic estimates, and deduces
Theorem~\ref{thm:intro-anisotropic-maximality}.

\subsection*{AI Disclosure}
The author used ChatGPT (OpenAI) for proofreading and bibliographic
searches during the preparation of this manuscript. The author also used ChatGPT for mathematical discussions, the exploration of proof strategies for the
projection--incidence and exceptional parameters theorem
(Theorem~\ref{thm:quantitative-line-slicing}) and generated
some of the examples appearing in the remarks. The author takes full
responsibility for all mathematical statements and proofs, the
examples, the accuracy of the references, and the final text.

\medskip
\section{Finite-ring and geometric preliminaries}
\label{sec:lines}
In this section, we prove the line-segment estimate needed for
Theorem~\ref{thm:padic}. We first state Dhar's finite-ring estimate
and deduce its restricted-direction form. We then pass from residue
rings to $\Q_p$. Finally, we prove a product dimension identity
that will be used to remove the auxiliary coordinates in the lift.
Dhar's set estimate will also be applied in
Section~\ref{sec:anisotropic-covering}. The tube maximal estimate
needed for individual slices is derived from the same set estimate
in Section~\ref{sec:measure-theoretic-input}.

\subsection*{Metrics and the finite-ring input}
For $x\in\Q_p^m$ and $\rho>0$, write
$B(x,\rho)=\{y\in\Q_p^m:\|y-x\|\leq\rho\}$ for the closed ball.
Haar measure $\mathcal L^m$ is normalized by
$\mathcal L^m(\Z_p^m)=1$, and integrals with respect to $dx$ use
this measure. We write $\HH^s$ for $s$-dimensional Hausdorff
measure and $\#$ for cardinality.

For nonzero $u,v\in\Q_p^m$, the projective metric is
\begin{equation*}
 \rho([u],[v])\defeq
 \frac{\max_{i,j}|u_iv_j-u_jv_i|_p}{\|u\|\|v\|}.
\end{equation*}
On a chart with representatives in $\Z_p^m$ and one fixed
coordinate equal to $1$, this is the supremum metric on the
remaining coordinates. In these normalizations,
$\dimH\Q_p^m=m$ and $\dimH\PP^{m-1}(\Q_p)=m-1$.

For $\ell\geq1$, put $\mathcal R_\ell\defeq \Z/p^\ell\Z$.
A vector in $\mathcal R_\ell^m$ is called \emph{primitive} if it has a unit
coordinate. The space $\PP^{m-1}(\mathcal R_\ell)$ consists of primitive
vectors modulo multiplication by $\mathcal R_\ell^\times$. A line of
direction $[u]$ is $a+\mathcal R_\ell u$.
Since $u$ has a unit coordinate, the parametrization
$t\mapsto a+tu$ is injective. Counting the primitive vectors and
then dividing by the number of units, we obtain
\begin{equation}\label{eq:direction-count}
 \#\PP^{m-1}(\mathcal R_\ell)
 =\frac{p^{\ell m}-p^{(\ell-1)m}}{p^\ell-p^{\ell-1}}
 \asymp_{p,m}p^{\ell(m-1)}.
\end{equation}

\begin{theorem}[Dhar's prime-power set estimate]
\label{thm:dhar}
Let $m,\ell\geq1$, $1\leq h\leq p^\ell$, and
$0<\varepsilon\leq1$. Suppose that $A\subset\mathcal R_\ell^m$
meets a line in at least $h$ points in each of at least
$\varepsilon\,\#\PP^{m-1}(\mathcal R_\ell)$ projective directions. Then
\begin{equation}\label{eq:dhar-set}
 \#A\geq
 \frac{\varepsilon\,h^m}
 {[2(\ell+\lceil\log_p m\rceil)]^m}.
\end{equation}
This is \cite[Theorem~1.8]{DharSet}, with the residue exponent denoted
by $\ell$.
\end{theorem}

We shall use the following consequence, in which the hypothesis
is stated in terms of the number of directions.

\begin{corollary}[Finite-scale restricted-direction estimate]
\label{cor:dhar-directions}
Let $m,\ell\geq1$ and $1\leq h\leq p^\ell$. Suppose that
$A\subset\mathcal R_\ell^m$ meets a line in at least
$h$ points in each of $J$ distinct projective directions.
Then
\begin{equation}\label{eq:dhar-J}
 \#A\geq c_{p,m}(\ell+1)^{-m}
 \frac{Jh^m}{p^{\ell(m-1)}}.
\end{equation}
\end{corollary}
\begin{proof}
If $J=0$, the inequality is immediate. Suppose that $J>0$.
Apply \eqref{eq:dhar-set} with
$\varepsilon=J/\#\PP^{m-1}(\mathcal R_\ell)$ and substitute
\eqref{eq:direction-count}. This gives the stated bound.
\end{proof}

\subsection*{From residue rings to line segments}
We now deduce a Hausdorff dimension estimate from the preceding
corollary. The proof is inspired by Arsovski's proof of the $p$-adic Kakeya conjecture \cite{Arsovski}.

\begin{theorem}[Restricted-direction line-segment estimate over $\Q_p$]
\label{thm:Qp-lines}
Let $m\geq1$ and let $D\subset\PP^{m-1}(\Q_p)$ be nonempty.
Suppose that for every $\theta\in D$ there are $a_\theta\in\Q_p^m$ and
$v_\theta\in\Q_p^m\setminus\{0\}$ with $[v_\theta]=\theta$ such that
\begin{equation*}
 a_\theta+\Z_pv_\theta\subset A\subset\Q_p^m.
\end{equation*}
Then
\begin{equation}\label{eq:Qp-lines}
 \dimH A\geq\dimH D+1.
\end{equation}
\end{theorem}

\begin{proof}
We first prove the theorem when $A\subset\Z_p^m$ and the
prescribed segments have unit length. Write these segments as
\begin{equation*}
 I_\theta\defeq \{a_\theta+tu_\theta:t\in\Z_p\}\subset A,
 \quad a_\theta\in\Z_p^m,\quad\norm{u_\theta}=1,\quad[u_\theta]=\theta, \quad \theta\in D.
\end{equation*}
Fix $\alpha>\dimH A$, an integer $\ell_0\geq1$, and $\varepsilon_0>0$.
Choose a cover of $A$ by balls of radii $p^{-\ell}$, $\ell\geq \ell_0$, satisfying
\begin{equation*}
 \sum_{\ell\geq \ell_0}N_\ell p^{-\ell\alpha}<\varepsilon_0,
\end{equation*}
where $N_\ell$ is the number of distinct balls at scale $p^{-\ell}$. Let $B_\ell$ be their
union, so $A\subset\bigcup_{\ell\geq \ell_0}B_\ell$.

For each $\theta\in D$, define
\begin{equation*}
 T_{\theta,\ell}\defeq \{t\in\Z_p:a_\theta+tu_\theta\in B_\ell\},\qquad
 \beta_\ell\defeq \frac{1}{2\ell(\ell+1)}.
\end{equation*}
Observe that $T_{\theta,\ell}$ is a finite union of parameter
residue classes. Since the sets $B_\ell$ cover $I_\theta$, we have
\begin{equation*}
 1\leq\sum_{\ell\geq \ell_0}\mathcal L^1(T_{\theta,\ell}),\qquad
 \sum_{\ell\geq1}\beta_\ell=\frac12.
\end{equation*}
Consequently, every direction belongs to at least one of the sets
\begin{equation*}
 D_\ell\defeq \{\theta\in D:\mathcal L^1(T_{\theta,\ell})\geq\beta_\ell\},\qquad
 D=\bigcup_{\ell\geq \ell_0}D_\ell.
\end{equation*}

Write $\bar x$ for the reduction modulo $p^\ell$ of an integral
vector $x\in\Z_p^m$. For projective directions, write
$\red_\ell$ for reduction using a primitive integral representative.
This is independent of the representative, since any two such
representatives differ by a unit. Put
\begin{equation*}
 S_\ell\defeq B_\ell\bmod p^\ell\subset \mathcal R_\ell^m,
 \qquad J_\ell\defeq \#\red_\ell(D_\ell).
\end{equation*}
By definition, $\#S_\ell=N_\ell$. For each $\theta\in D_\ell$,
consider the line over $\mathcal R_\ell$ given by
\begin{equation*}
 L_{\theta,\ell}\defeq \{\overline{a_\theta}+\bar t\,\overline{u_\theta}:\bar t\in \mathcal R_\ell\}.
\end{equation*}
The parameter map is injective because $u_\theta$ has a unit
coordinate. Moreover, membership in $B_\ell$ depends only on
$t\bmod p^\ell$. Each such parameter class has measure
$p^{-\ell}$, and hence
\begin{equation*}
 \#(S_\ell\cap L_{\theta,\ell})=p^\ell\mathcal L^1(T_{\theta,\ell})\geq\beta_\ell p^\ell.
\end{equation*}
Choose one direction in $D_\ell$ for each distinct reduction.
Applying \eqref{eq:dhar-J} with
$h=\lceil\beta_\ell p^\ell\rceil$, we obtain
\begin{align*}
 N_\ell
 &\geq c_{p,m}(\ell+1)^{-m}
 \frac{J_\ell(\beta_\ell p^\ell)^m}{p^{\ell(m-1)}}\notag\\
 &\geq c'_{p,m}(\ell+1)^{-3m}J_\ell p^\ell.
\end{align*}
Therefore
\begin{equation}\label{eq:direction-cover-count}
 J_\ell\leq C_{p,m}(\ell+1)^{3m}N_\ell p^{-\ell}.
\end{equation}

We next turn this count into a cover of the direction set.
Recall that the reduction map
\begin{equation*}
 \red_\ell:\PP^{m-1}(\Q_p)\longrightarrow\PP^{m-1}(\mathcal R_\ell)
\end{equation*}
is defined using any primitive integral representative. Its fibers are
projective balls of radius $p^{-\ell}$. Indeed, in a fiber one can choose a common
unit coordinate and normalize it to $1$. Equality of reductions then means
that all remaining coordinates agree modulo $p^\ell$. Thus $D_\ell$ is covered by
$J_\ell$ such balls, and their union over $\ell$ covers $D$.

Fix $\eta>0$, and write $\xi\defeq \alpha-1+\eta$. Because $A$ contains a unit segment,
$\dimH A\geq1$, so $\xi>0$. By \eqref{eq:direction-cover-count}, the
$\xi$-dimensional cost of the direction cover is at most
\begin{align*}
 \sum_{\ell\geq \ell_0}J_\ell p^{-\ell\xi}
 &\leq C_{p,m}\sum_{\ell\geq \ell_0}
 (\ell+1)^{3m}N_\ell p^{-\ell(\xi+1)}\notag\\
 &=C_{p,m}\sum_{\ell\geq \ell_0}
 (\ell+1)^{3m}p^{-\ell\eta}N_\ell p^{-\ell\alpha}\notag\\
 &\leq C_{p,m,\eta}\sum_{\ell\geq \ell_0}N_\ell p^{-\ell\alpha}
 <C_{p,m,\eta}\varepsilon_0.
\end{align*}
The last inequality uses
$\sup_{\ell\geq1}(\ell+1)^{3m}p^{-\ell\eta}<\infty$.
Since $\varepsilon_0$ and $p^{-\ell_0}$ may be arbitrarily small,
it follows that $\HH^{\alpha-1+\eta}(D)=0$.
Letting $\alpha\downarrow\dimH A$ and then $\eta\downarrow0$
proves the assertion for unit segments in $\Z_p^m$.

It remains to remove the restrictions on the lengths and locations
of the segments. Partition $\Q_p^m$ into its countably many
unit-ball cosets $c_\nu+\Z_p^m$, and define
\begin{equation*}
 A_*\defeq
 \bigcup_{h\in\Z}\ \bigcup_\nu
 \bigl((p^hA\cap(c_\nu+\Z_p^m))-c_\nu\bigr)\subset\Z_p^m.
\end{equation*}
Each member of this countable union is a translated subset of a similarity
image of $A$, so $\dimH A_*\leq\dimH A$.
For a prescribed segment $a_\theta+\Z_pv_\theta$, choose $h\in\Z$ such that
$\norm{v_\theta}=p^h$. Then $\norm{p^hv_\theta}=1$, and the segment
$p^ha_\theta+\Z_p(p^hv_\theta)$ lies in one unit-ball coset.
After translation, this is a unit segment in $A_*$ with direction
$\theta$. Applying the result already proved, we obtain
\begin{equation*}
 \dimH A\geq\dimH A_*\geq\dimH D+1.
\end{equation*}
This completes the proof.
\end{proof}

\subsection*{Products with full-dimensional factors}
In the lifting arguments, the direction sets and fibers contain
full $\Q_p$ factors. We shall use the following lemma to calculate
the contribution of these factors.

\begin{lemma}[Full $\Q_p$ factors]
\label{lem:products}
For every nonempty $X\subset \Q_p^a$ and every integer $b\geq0$,
\begin{equation}\label{eq:product-dimension}
 \dimH(X\times \Q_p^b)=\dimH X+b.
\end{equation}
The same equality holds with the second factor replaced by any nonempty open
subset of $\Q_p^b$.
\end{lemma}

\begin{proof}
The case $b=0$ is immediate, so we assume $b>0$. By countable stability and bi-Lipschitz invariance of
Hausdorff dimension, it suffices to treat the factor $\Z_p^b$, because 
every nonempty open subset of $\Q_p^b$ is a countable union of
balls, and $X\times B$ is bi-Lipschitz equivalent to
$X\times\Z_p^b$ for every such ball $B$.

For $s\geq0$ and $0<\rho\leq1$, let $\mathcal H_\rho^s$ denote
Hausdorff content computed using balls of positive radii
$r\in p^{\Z}$, $r\leq\rho$, with cost $r^s$.
Given a ball cover $\{B_i\}$ of $X$, with radii $r_i\leq\rho$,
partition $\Z_p^b$ into $r_i^{-b}$ balls of radius $r_i$.
The resulting product cover gives
\begin{equation*}
 \mathcal H_\rho^{s+b}(X\times\Z_p^b)
 \leq \inf_{\{B_i\}}\sum_i r_i^{-b}r_i^{s+b}
 =\mathcal H_\rho^s(X).
\end{equation*}

Conversely, let $\{A_i\times C_i\}$ be a ball cover of
$X\times\Z_p^b$, where both factors have radius $r_i\leq\rho$.
For each $y\in\Z_p^b$, the balls $A_i$ with $y\in C_i$ cover $X$.
Integrating against $\mathcal{L}^b$ yields
\begin{equation*}
 \mathcal H_\rho^s(X)
 \leq \int_{\Z_p^b}\sum_i r_i^s\mathbf 1_{C_i}(y)\,d\mathcal{L}^b(y)
 \leq \sum_i r_i^{s+b}.
\end{equation*}
Taking the infimum over these covers proves the reverse inequality.
Hence
\begin{equation*}
 \mathcal H_\rho^{s+b}(X\times\Z_p^b)
 =\mathcal H_\rho^s(X).
\end{equation*}
Letting $\rho\downarrow0$ and varying $s$ gives
$\dimH(X\times\Z_p^b)=\dimH X+b$; when $\dimH X=0$,
the lower bound follows by taking $s=0$, since $X\neq\varnothing$.
\end{proof}

\medskip
\section{Polynomial-to-line lifting and the curved Kakeya theorem}
\label{sec:lifting}

We devote this section to the proof of Theorem~\ref{thm:padic}.
First, we show that the inverse image of $E$ under the lifting map
contains line segments in a suitable set of directions. Next, we
compute the dimension of this inverse image. We then use unit
reparametrizations to pass from the prescribed leading directions
to leading vectors. The theorem will follow by applying the line
estimate and comparing these dimensions.

\subsection*{The lift and its fibers}
Fix integers $n,d\geq1$. Define
\begin{equation*}
 \Phi:(\Q_p^n)^d\times \Q_p\longrightarrow \Q_p^n,\qquad
 \Phi(x_0,\ldots,x_{d-1},z)\defeq \sum_{j=0}^{d-1}z^jx_j.
\end{equation*}
For $B\subset \Q_p^n$, write $K_B\defeq \Phi^{-1}(B)\subset \Q_p^{dn+1}$.
The following calculation is Nadjimzadah's algebraic construction
\cite[Section~4, equations~(4.11)--(4.16)]{Nadjimzadah}, with
the parameter omitted from the output. We give the calculation
in our notation.

\begin{lemma}[Line segments in the polynomial lift]
\label{lem:lifting}
Suppose that for every $w$ in a set $W\subset\Q_p^n\setminus\{0\}$
there is a polynomial map
\begin{equation*}
 \gamma_w(t)\defeq a_0(w)+a_1(w)t+\cdots+a_{d-1}(w)t^{d-1}+wt^d,
 \qquad \gamma_w(\Z_p)\subset B.
\end{equation*}
Then $K_B$ contains a nondegenerate affine line segment in every direction in
\begin{equation}\label{eq:lift-directions}
 D(W)\defeq \bigl\{
 [u_0:\cdots:u_{d-2}:w:1]:
 u_j\in \Q_p^n,\ w\in W
 \bigr\}\subset\PP^{dn}(\Q_p).
\end{equation}
When $d=1$, the list $u_0,\ldots,u_{d-2}$ is empty.
\end{lemma}

\begin{proof}
Fix $w\in W$, choose arbitrary $u_0,\ldots,u_{d-2}\in \Q_p^n$, and put
$u_{d-1}\defeq w$. Define
\begin{equation*}
 b_0\defeq a_0(w),\qquad b_j\defeq a_j(w)-u_{j-1}\quad(1\leq j\leq d-1).
\end{equation*}
Consider the line
\begin{equation*}
 \mathcal L(t)\defeq (b_0+tu_0,\ldots,b_{d-1}+tu_{d-1},t),\qquad t\in \Q_p.
\end{equation*}
We claim that it has the required property. Its last coordinate is $t$, so it is
nondegenerate. Substituting into the definition of $\Phi$, we have
\begin{align*}
 \Phi(\mathcal L(t))
 &=\sum_{j=0}^{d-1}t^j(b_j+tu_j)\notag\\
 &=b_0+\sum_{j=1}^{d-1}(b_j+u_{j-1})t^j+u_{d-1}t^d\notag\\
 &=\gamma_w(t).
\end{align*}
Since $\gamma_w(\Z_p)\subset B$, this implies
\begin{equation*}
 \mathcal L(\Z_p)\subset K_B,
 \qquad \mathrm{direction}\,\mathcal L=[u_0:\cdots:u_{d-2}:w:1].
\end{equation*}
The direction vector has norm at least $1$. Thus, by restricting
$t$ to a suitable ball $p^h\Z_p\subset\Z_p$, the segment
contains a unit segment. The vectors $u_0,\ldots,u_{d-2}$ were
arbitrary, so every stated direction occurs. This proves the lemma.
\end{proof}

The final homogeneous coordinate in \eqref{eq:lift-directions} is
fixed to $1$, so distinct tuples determine distinct projective
directions. This parametrization is locally bi-Lipschitz and identifies
the direction set with
\begin{equation*}
 D(W)\simeq \Q_p^{(d-1)n}\times W.
\end{equation*}
We next describe the full inverse image $K_B$. The following
coordinate change identifies it with a product whose first factor
is $B$.

\begin{lemma}[The fibers of the lift]
\label{lem:automorphism}
The polynomial map
\begin{equation}\label{eq:psi}
 \Psi(e,x_1,\ldots,x_{d-1},z)
 \defeq \left(e-\sum_{j=1}^{d-1}z^jx_j,x_1,\ldots,x_{d-1},z\right)
\end{equation}
is an automorphism of $\Q_p^{dn+1}$, with inverse
\begin{equation}\label{eq:psi-inverse}
 \Psi^{-1}(x_0,\ldots,x_{d-1},z)
 =\left(\sum_{j=0}^{d-1}z^jx_j,x_1,\ldots,x_{d-1},z\right).
\end{equation}
It restricts to a bijection
\begin{equation*}
 \Psi:B\times \Q_p^{(d-1)n+1}\longrightarrow K_B.
\end{equation*}
\end{lemma}

\begin{proof}
Fix $e,x_1,\ldots,x_{d-1},z$. The equation defining the fiber
has a unique solution for $x_0$, given by \eqref{eq:psi}.
Substitution shows that \eqref{eq:psi-inverse} is its inverse.
The stated bijection follows.
\end{proof}

To pass from this bijection to a dimension identity, it remains
to check that both coordinate maps are Lipschitz on bounded boxes.
We include this verification in the next lemma.

\begin{lemma}[Hausdorff dimension of the polynomial lift]
\label{lem:lift-dimension}
For nonempty $B\subset \Q_p^n$,
\begin{equation}\label{eq:lift-dimension}
 \dimH K_B=\dimH B+(d-1)n+1.
\end{equation}
\end{lemma}
\begin{proof}
By Lemma~\ref{lem:automorphism}, the map \eqref{eq:psi} sends
$B\times \Q_p^{(d-1)n+1}$ onto $K_B$. Suppose that the
coordinates of two inputs have norm at most $R\geq1$, and their
supremum distance is $\delta$. For $j\geq1$, we have
\begin{equation*}
 z^j-w^j=(z-w)\sum_{j'=0}^{j-1}z^{j-1-j'}w^{j'},
 \qquad |z^j-w^j|_p\leq R^{j-1}|z-w|_p.
\end{equation*}
Adding and subtracting $z^jy_j$, then using the ultrametric inequality, yields
\begin{align*}
 \norm{z^jx_j-w^jy_j}
 &=\norm{z^j(x_j-y_j)+(z^j-w^j)y_j}\notag\\
 &\leq\max\bigl\{|z|_p^j\norm{x_j-y_j},
 |z^j-w^j|_p\norm{y_j}\bigr\}
 \leq R^j\delta.
\end{align*}
It follows that $\norm{\Psi(P)-\Psi(Q)}\leq R^{d-1}\norm{P-Q}$ on that box.
The inverse satisfies the same kind of estimate. Taking the countable union of bounded boxes covering $\Q_p^{dn+1}$, we obtain
\begin{equation*}
 \dimH K_B=\dimH(B\times \Q_p^{(d-1)n+1}).
\end{equation*}
By \eqref{eq:product-dimension}, the right-hand side is
$\dimH B+(d-1)n+1$. This proves \eqref{eq:lift-dimension}.
\end{proof}

\subsection*{Leading directions and unit reparametrizations}
The lifting lemma is stated in terms of leading vectors. In
Theorem~\ref{thm:padic}, however, only their projective directions
are prescribed. We now explain how to choose sufficiently many
leading vectors without changing the polynomial images. First,
we compute the dimension of the cone over a direction set.

\begin{lemma}[Dimension of the directional cone]
\label{lem:cone}
For nonempty $V\subset\PP^{n-1}(\Q_p)$, let
\begin{equation*}
 C(V)\defeq \{v\in \Q_p^n\setminus\{0\}:[v]\in V\}.
\end{equation*}
Then
\begin{equation}\label{eq:cone-dimension}
 \dimH C(V)=\dimH V+1.
\end{equation}
\end{lemma}
\begin{proof}
Fix a projective chart, normalize one coordinate to $1$, and
write the resulting representative as $v(\omega)$. The map
$(\omega,\lambda)\mapsto\lambda v(\omega)$, with $\lambda\in \Q_p^\times$,
parametrizes the corresponding part of the cone. On bounded pieces where
$|\lambda|_p$ is bounded away from zero, this map and its inverse are
Lipschitz. By Lemma~\ref{lem:products}, the nonzero scalar factor contributes
one dimension. Taking the countable union over bounded chart
pieces and scalar annuli gives \eqref{eq:cone-dimension}.
\end{proof}

\begin{lemma}[Leading vectors from unit reparametrizations]
\label{lem:representatives}
Choose one polynomial $\gamma_\omega$ and its leading vector $v_\omega$
for each $\omega\in V$ as in the curved Kakeya theorem, and put
\begin{equation*}
 W_d\defeq \{\alpha^d v_\omega:\omega\in V,\ \alpha\in\Z_p^\times\}.
\end{equation*}
Every $w\in W_d$ is the leading coefficient of a degree-$d$ polynomial
$\gamma_w$ satisfying $\gamma_w(\Z_p)\subset E$. Moreover,
\begin{equation*}
 \dimH W_d=\dimH V+1.
\end{equation*}
\end{lemma}

\begin{proof}
If $w=\alpha^d v_\omega$ with $\alpha\in\Z_p^\times$, take
\begin{equation*}
 \gamma_w(t)\defeq \gamma_\omega(\alpha t)
 =\sum_{j=0}^{d-1}\alpha^j a_{\omega,j}t^j+wt^d.
\end{equation*}
Since $\alpha\Z_p=\Z_p$,
\begin{equation*}
 \gamma_w(\Z_p)=\gamma_\omega(\alpha\Z_p)
 =\gamma_\omega(\Z_p)\subset E.
\end{equation*}
This proves the first assertion.

It remains to calculate the dimension of $W_d$. We first show
that $\Z_p^\times/(\Z_p^\times)^d$ is finite. Write $v_p$ for the additive
valuation with $v_p(p)=1$, and put $h_0\defeq 2v_p(d)+1$. If
$u\in1+p^{h_0}\Z_p$, apply the strong form of Hensel's lemma to
$f(X)\defeq X^d-u$ at $X=1$. Indeed,
\begin{equation*}
 v_p(f(1))\geq h_0>2v_p(d)=2v_p(f'(1)).
\end{equation*}
The strong Hensel criterion \cite[Theorem~7.32]{Milne} gives a root
$\beta\in\Z_p$ with $\beta^d=u$. It is a unit since $u$ is a unit.
Consequently,
\begin{equation*}
 1+p^{h_0}\Z_p\subset(\Z_p^\times)^d.
\end{equation*}
Since $\Z_p^\times/(1+p^{h_0}\Z_p)$ is finite, the claim
follows. The argument also applies when $p\mid d$.

Choose representatives $c_1,\ldots,c_J\in\Z_p^\times$ for
$\Z_p^\times/(\Z_p^\times)^d$. Since
$\Q_p^\times=p^{\Z}\Z_p^\times$, the directional cone satisfies
\begin{equation*}
 C(V)=\bigcup_{h\in\Z}\ \bigcup_{i=1}^J p^hc_iW_d.
\end{equation*}
Indeed, every $v\in C(V)$ has the form $\lambda v_\omega$, and
$\lambda=p^hc_i\alpha^d$ for some $h\in\Z$, $1\leq i\leq J$, and
$\alpha\in\Z_p^\times$. Conversely, scalar multiplication preserves
the projective direction. So countable
stability and \eqref{eq:cone-dimension} give
\begin{equation*}
 \dimH C(V)
 =\sup_{h\in\Z,\,1\leq i\leq J}\dimH(p^hc_iW_d)
 =\dimH W_d=\dimH V+1.
\end{equation*}
This proves the dimension assertion. 
\end{proof}

\subsection*{The dimension bound and its scope}
We are now ready to prove Theorem~\ref{thm:padic}. The preceding
lemmas give the dimensions of the lifted direction set and the
full inverse image. Comparing them by the line theorem will give
the desired lower bound.

\begin{proof}[Proof of Theorem~\ref{thm:padic}]
Let $s\defeq \dimH V$, and let $W_d$ be the set from
Lemma~\ref{lem:representatives}. Then $\dimH W_d=s+1$.
For each $w\in W_d$, the lemma gives a polynomial with leading
vector $w$ whose image satisfies $\gamma_w(\Z_p)\subset E$.
We may therefore apply Lemma~\ref{lem:lifting} with $B=E$ and
$W=W_d$.
It gives a nondegenerate affine line segment in $K_E\defeq \Phi^{-1}(E)$
in every direction in
\begin{equation*}
 D(W_d)=\{[u_0:\cdots:u_{d-2}:w:1]:
 u_j\in\Q_p^n,\ w\in W_d\}.
\end{equation*}
The projective parametrization is locally bi-Lipschitz in its affine chart.
Consequently, Lemma~\ref{lem:products} and the dimension of $W_d$ give
\begin{equation*}
 \dimH D(W_d)=(d-1)n+\dimH W_d
 =(d-1)n+s+1.
\end{equation*}
The restricted-direction segment estimate \eqref{eq:Qp-lines}, applied
in ambient dimension $dn+1$, yields
\begin{equation*}
 \dimH K_E\geq\dimH D(W_d)+1=(d-1)n+s+2.
\end{equation*}
On the other hand, Lemma~\ref{lem:lift-dimension} gives
$\dimH K_E=\dimH E+(d-1)n+1$. This computes the dimension
of the full inverse image $\Phi^{-1}(E)$. It does not require
that the lifted segments extend to entire lines in that set.
Combining the two bounds, we obtain
\begin{equation*}
 (d-1)n+s+2\leq\dimH K_E=\dimH E+(d-1)n+1.
\end{equation*}
Subtracting the common auxiliary dimension gives
$\dimH E\geq s+1$, as required.
\end{proof}

\begin{remark}[Sharpness]
\label{rem:sharpness-and-scope}
For every nonempty $V\subset\PP^{n-1}(\Q_p)$, the bounded set
\begin{equation*}
 E_V\defeq \{0\}\cup\bigl(C(V)\cap\Z_p^n\bigr)
\end{equation*}
contains $\gamma_\omega(\Z_p)$ for
$\gamma_\omega(t)\defeq v_\omega t^d$ with $\norm{v_\omega}=1$.
Since $C(V)=\bigcup_{h\geq0}p^{-h}(E_V\setminus\{0\})$, countable
stability and the cone lemma give $\dimH E_V=\dimH V+1$.
Thus the lower bound in Theorem~\ref{thm:padic} is attained for
each prescribed $V$.
\end{remark}

\begin{remark}[Finite field extensions]
\label{rem:finite-extensions}
The curved Kakeya theorem also holds over a finite extension
$F/\Q_p$, with ring of integers $\mathcal O_F$, when the hypothesis
is $\gamma_\omega(\mathcal O_F)\subset E$. To adapt the proof,
we first obtain the required line estimate by restriction of scalars.
Normalize $|x|_F\defeq |N_{F/\Q_p}(x)|_p$, so that $\dimH F=1$.
Put $\kappa\defeq [F:\Q_p]$ and identify $F^n$ with
$\Q_p^{\kappa n}$ using a $\Q_p$-basis of $F$. The coordinate
metric is comparable to the $1/\kappa$ power of the normalized
$F$-metric. Write $\dimH^{\Q_p}$ for Hausdorff dimension in
$\Q_p$ coordinates and $\dimH^F$ for Hausdorff dimension in the
normalized $F$-metric. Thus $\dimH^{\Q_p}=\kappa\dimH^F$
for sets viewed in these two metrics.

Suppose that $A\subset F^n$ contains nondegenerate $F$-segments in
a nonempty set of directions $D\subset\PP^{n-1}(F)$, and let
$\widetilde D\subset\PP^{\kappa n-1}(\Q_p)$ be the corresponding
set of $\Q_p$-line directions. The $\Q_p$-line directions inside
each $F$-direction form $\PP^{\kappa-1}(\Q_p)$. Local projective
product charts and Lemma~\ref{lem:products} therefore give
\begin{equation*}
 \dimH^{\Q_p}\widetilde D
 =\kappa\dimH^F D+(\kappa-1).
\end{equation*}
Moreover, every nondegenerate $F$-segment contains a nondegenerate
$\Q_p$-segment in each such direction: multiply a given scalar by
a sufficiently large power of $p$ to place its $\Z_p$-multiples
inside the parameter ball. The segment form of \eqref{eq:Qp-lines}
now yields
\begin{equation*}
 \kappa\dimH^F A=\dimH^{\Q_p}A
 \geq\dimH^{\Q_p}\widetilde D+1
 =\kappa\dimH^F D+\kappa.
\end{equation*}
Thus $\dimH^F A\geq\dimH^F D+1$, as required for the line estimate
over $F$.

Unit reparametrizations preserve $\mathcal O_F$, and the same Hensel argument
shows that $\mathcal O_F^\times/(\mathcal O_F^\times)^d$ is finite.
Using powers of a uniformizer in the countable cone decomposition, the
leading-vector lemma and the lifting proof are unchanged. They give
$\dimH E\geq\dimH V+1$ under $\gamma_\omega(\mathcal O_F)\subset E$.
\end{remark}

\medskip
\section{Measure-theoretic input and projection incidence}
\label{sec:measure-theoretic-input}

Theorem~\ref{thm:padic} gives a lower bound for a union of curves.
To study evaluation images, we also need a lower bound for its
individual slices. In this section, we prove such a statement for
analytic families of lines. We first choose measures on the slices
using Frostman's lemma and measurable selection. We then derive
a tube maximal inequality from Dhar's set estimate and apply it to
estimate their small-ball masses.
This gives Theorem~\ref{thm:quantitative-line-slicing}, which will
be used in Sections~\ref{sec:polynomial-projections}
and~\ref{sec:block-evaluations}.

\subsection*{Frostman measures and measurable selection}
We begin by stating the measure-theoretic results needed in the
proof. A subset of a Polish space is
\emph{universally measurable} if it belongs to the completion of every
finite Borel measure on that space. A finite Borel measure is
\emph{carried by} such a set if its completion assigns full mass to
the set. Analytic sets are universally measurable \cite{Kechris}, so
this terminology applies even when the sets in question are not Borel.

\begin{lemma}[Frostman measures on analytic subsets of $\Q_p^m$]
\label{lem:projection-frostman}
Let $m\geq1$. If $A\subset\Q_p^m$ is analytic and $0<\tau<\dimH A$, then
there are a compact set $K\subset A$, a Borel probability measure
$\nu$ carried by $K$, and a constant $C>0$ such that
\begin{equation*}
 \nu(B(u,\rho))\leq C\rho^\tau
 \qquad(u\in\Q_p^m,\ 0<\rho\leq1).
\end{equation*}
\end{lemma}
\begin{proof}
By countable stability of Hausdorff dimension, we may choose a
bounded ball $B_0$ such that $\dimH(A\cap B_0)>\tau$.
Up to a similarity, $B_0$ is the boundary of its residue-class
tree. Let $d_p$ be the
natural tree metric with scale $p^{-j}$, and let $d_2$ be the
metric with scale $2^{-j}$ used in \cite{BP}. These metrics satisfy
\begin{equation*}
 d_p=d_2^{\log_2 p}.
\end{equation*}
Since $A\cap B_0$ is analytic and has positive $\tau$-dimensional
Hausdorff measure in $d_p$, it has positive
$(\tau\log_2 p)$-dimensional Hausdorff measure in $d_2$.
Applying \cite[Corollary~B.2.4]{BP} with exponent $\tau\log_2 p$
gives a compact $K\subset A\cap B_0$ with positive measure of that
dimension in $d_2$, and hence $\HH^\tau(K)>0$ in the original
$p$-adic metric.

Frostman's lemma for compact metric spaces
\cite[Theorem~8.17]{Mattila} supplies a nonzero finite Borel
measure $\mu$ on $K$ which, extended by zero to $\Q_p^m$, satisfies
\begin{equation*}
  \mu(B(u,\rho))\leq C_0\rho^\tau
\end{equation*}
for every $u\in\Q_p^m$ and all sufficiently small $\rho>0$.
Normalize $\mu$ to have total mass one. By increasing the
constant, the same estimate holds for all remaining radii
$0<\rho\leq1$. The resulting measure is the required $\nu$.
\end{proof}

We shall see for an analytic line family, the slope set is also analytic.
The preceding lemma therefore gives a Frostman measure on a
compact subset of the slopes. We must now lift this measure to
the line family while retaining its slope marginal. The following
selection theorem allows us to choose one intercept for each
slope.

\begin{theorem}[Jankov--von Neumann measurable selection]
\label{thm:projection-selection}
Let $X,Z$ be Polish spaces, and let $Y\subset X\times Z$ be analytic.
Writing $D\defeq \{u\in X:\exists w\in Z,\ (u,w)\in Y\}$, there is a universally
measurable map $\varsigma:D\to Z$ with $(u,\varsigma(u))\in Y$ for every $u\in D$.
Consequently, any Borel probability measure $\nu$ carried by $D$ lifts to a
Borel probability measure $\lambda$ carried by $Y$ with first-coordinate
marginal $\nu$. This is the universal-measurability consequence of
\cite[Theorem~18.1]{Kechris}.
\end{theorem}

For completeness, the lifted measure is defined by
\begin{equation*}
 \lambda\defeq (u\mapsto(u,\varsigma(u)))_*\nu.
\end{equation*}
The map is defined on $D$, which has full $\nu$-measure.
We use the completion of $\nu$ to define the pushforward on
Borel sets. This is well-defined by universal measurability.
The map takes values in $Y$ and preserves the first coordinate,
so $\lambda$ is carried by $Y$ and has marginal $\nu$.

We shall apply this construction to a bounded part of the line
family, with $\nu$ supplied by
Lemma~\ref{lem:projection-frostman}. Evaluating the lines gives
a probability measure on each slice. We then estimate the
mass of small balls for these measures to invoke the mass distribution principle to obtain a lower bound on the Hausdorff dimension. For this purpose we use
the following maximal inequality.

\subsection*{The tube maximal estimate}
For integers $m\geq2$ and $\ell\geq1$, put $\delta=p^{-\ell}$.
For $f\in L^m(\Z_p^m)$, define
\begin{equation*}
 \mathcal K_\delta f(u)
 \defeq \sup_{w\in\Z_p^{m-1}}
 \int_{\Z_p}\frac{1}{\delta^{m-1}}
 \int_{B(w+tu,\delta)}|f(t,z)|\,dz\,dt,
 \qquad u\in\Z_p^{m-1}.
\end{equation*}
The graph $t\mapsto(t,w+tu)$, $t\in\Z_p$, is a unit line
segment, and its $\delta$-neighborhood has Haar measure
$\delta^{m-1}$. Thus the factor $\delta^{-(m-1)}$ normalizes
the integral to be the average over that tube.

For the following lemma, related finite-ring maximal estimates were obtained by Dhar
\cite[Theorem~1.6]{DharMax}. For the present argument, we use
only the published set estimate \cite[Theorem~1.8]{DharSet}
and give the required maximal-inequality deduction below.

\begin{lemma}[Tube maximal inequality from Dhar's set estimate]
\label{lem:projection-tube-maximal}
For every $m\geq2$, $\ell\geq1$, and $f\in L^m(\Z_p^m)$,
\begin{equation}\label{eq:projection-tube-maximal}
 \|\mathcal K_{p^{-\ell}}f\|_{L^m(\Z_p^{m-1})}
 \leq C_{p,m}(\ell+1)^2\|f\|_{L^m(\Z_p^m)}.
\end{equation}
In particular, for every $\varepsilon>0$ and $\delta=p^{-\ell}$,
\begin{equation*}
 \|\mathcal K_\delta f\|_{L^m(\Z_p^{m-1})}
 \leq C_{p,m,\varepsilon}\delta^{-\varepsilon}
            \|f\|_{L^m(\Z_p^m)}.
\end{equation*}
\end{lemma}
\begin{proof}
Put $q=p^\ell$, and let $\eta_k$ denote normalized counting
measure on $\mathcal R_\ell^k$. For a nonnegative function
$g:\mathcal R_\ell^m\to[0,\infty)$, define
\begin{equation*}
 \mathcal M_\ell g(u)
 \defeq\max_{w\in\mathcal R_\ell^{m-1}}
       \frac1q\sum_{t\in\mathcal R_\ell}g(t,w+tu),
 \qquad u\in\mathcal R_\ell^{m-1}.
\end{equation*}
We first establish the finite-ring bound with all norms taken
with respect to these probability measures.

Let $A\subset\mathcal R_\ell^m$ and $0<\alpha<1$.
If $\mathcal M_\ell\one_A(u)>\alpha$, then $A$ meets a line
of direction $[1:u]$ in at least $\lceil q\alpha\rceil$ points.
Distinct slopes give distinct projective directions. Applying
Corollary~\ref{cor:dhar-directions} to these directions and
rearranging gives
\begin{equation*}
 \eta_{m-1}\{u:\mathcal M_\ell\one_A(u)>\alpha\}
 \leq C_{p,m}(\ell+1)^m\alpha^{-m}\eta_m(A).
\end{equation*}
This is the restricted weak-type estimate obtained from the set
bound. Since $0\leq\mathcal M_\ell\one_A\leq1$, integration of
its distribution function gives, for nonempty $A$,
\begin{equation*}
    \begin{aligned}
 \|\mathcal M_\ell\one_A\|_{L^m(\eta_{m-1})}^m
 &=m\int_0^1\alpha^{m-1}
       \eta_{m-1}\{\mathcal M_\ell\one_A>\alpha\}\,d\alpha\\
 &\leq q^{-m}
       +C_{p,m}(\ell+1)^m\eta_m(A)
                         \int_{q^{-1}}^1\frac{d\alpha}{\alpha}\\
 &\leq C_{p,m}(\ell+1)^{m+1}\eta_m(A).
\end{aligned}
\end{equation*}
Here we used that $\eta_{m-1}$ is a probability measure,
$\log q=\ell\log p$, and $q^{-m}\leq\eta_m(A)$ for nonempty $A$.
The estimate is immediate when $A$ is empty. Thus
\begin{equation*}
 \|\mathcal M_\ell\one_A\|_{L^m(\eta_{m-1})}
 \leq C_{p,m}(\ell+1)^{1+1/m}\eta_m(A)^{1/m}.
\end{equation*}

For a general nonzero $g$, normalize
$\|g\|_{L^m(\eta_m)}=1$. Every point has $\eta_m$-mass
$q^{-m}$, so $\|g\|_\infty\leq q$.
Put $L=\lceil\log_2q\rceil$ and
\begin{equation*}
 A_j\defeq\{x:2^j<g(x)\leq2^{j+1}\},
 \qquad 0\leq j<L.
\end{equation*}
These sets are disjoint and satisfy
\begin{equation*}
 g\leq1+2\sum_{j=0}^{L-1}2^j\one_{A_j},
 \qquad
 \sum_{j=0}^{L-1}2^{jm}\eta_m(A_j)\leq1.
\end{equation*}
Sublinearity, the indicator-function estimate, and H\"older's
inequality in the finite sum imply
\begin{equation*}
\begin{aligned}
 \|\mathcal M_\ell g\|_{L^m(\eta_{m-1})}
 &\leq1+C_{p,m}(\ell+1)^{1+1/m}
                    \sum_{j=0}^{L-1}2^j\eta_m(A_j)^{1/m}\\
 &\leq1+C_{p,m}(\ell+1)^{1+1/m}L^{1-1/m}
       \left(\sum_{j=0}^{L-1}2^{jm}\eta_m(A_j)\right)^{1/m}\\
 &\leq C_{p,m}(\ell+1)^2,
\end{aligned}
\end{equation*}
since $L\leq C_p(\ell+1)$. By homogeneity, including the case
$g=0$, this proves
\begin{equation*}
 \|\mathcal M_\ell g\|_{L^m(\eta_{m-1})}
 \leq C_{p,m}(\ell+1)^2\|g\|_{L^m(\eta_m)}.
\end{equation*}

To pass to $\Z_p^m$, average $|f|$ over residue classes:
\begin{equation*}
 g(\bar x)\defeq q^m
       \int_{x+p^\ell\Z_p^m}|f(y)|\,dy,
 \qquad \bar x\in\mathcal R_\ell^m.
\end{equation*}
Jensen's inequality gives
$\|g\|_{L^m(\eta_m)}\leq\|f\|_{L^m(\Z_p^m)}$.
The tube in direction $[1:u]$ with intercept $w$ is the union
of the $q$ residue classes corresponding to
$(\bar t,\bar w+\bar t\bar u)$, with $\bar t\in\mathcal R_\ell$.
Its normalized integral is therefore
$q^{-1}\sum_{\bar t}g(\bar t,\bar w+\bar t\bar u)$.
Consequently,
\begin{equation*}
 \mathcal K_{p^{-\ell}}f(u)=\mathcal M_\ell g(\bar u),
 \qquad
 \|\mathcal K_{p^{-\ell}}f\|_{L^m(\Z_p^{m-1})}
   =\|\mathcal M_\ell g\|_{L^m(\eta_{m-1})}.
\end{equation*}
The finite-ring estimate proves
\eqref{eq:projection-tube-maximal}. Finally,
$(\ell+1)^2\leq C_{p,\varepsilon}p^{\varepsilon\ell}$ gives
the stated $\delta^{-\varepsilon}$ form.
\end{proof}

\subsection*{Projection incidence and exceptional parameters}
We now prove the projection--incidence theorem. The bound for the
union follows from Theorem~\ref{thm:Qp-lines}. For individual
slices, we choose measures as above and use duality to estimate an
averaged moment of their small-ball masses. In integral coordinates,
this moment is constant on parameter balls of the same radius.
The moment estimate therefore bounds the number of parameter balls
where the mass is too large. We shall show that the resulting
limsup set contains all exceptional parameters.

For related weighted-direction and duality arguments over
Euclidean spaces, see Mitsis \cite[Theorem~3.1]{Mitsis} and
Gauvan \cite[Theorem~5]{Gauvan}. In particular, Gauvan compares
maximal estimates for the ambient direction measure with estimates
for a Frostman measure. Here the required loss is obtained by
averaging the slope measure over residue classes. The use of
projected measures for affine-family slices is also related to
D.~M.~Oberlin's adjoint-transform formulation
\cite[Section~3]{Oberlin}. We do not invoke these Euclidean
estimates in the proof. The maximal input is
Lemma~\ref{lem:projection-tube-maximal}.

\Needspace{6\baselineskip}
\begin{theorem}[Projection--incidence and exceptional parameters]
\label{thm:quantitative-line-slicing}
Let $Y\subset\Q_p^{m-1}\times\Q_p^{m-1}$ be nonempty and analytic, where $m\geq2$.
Define the slope set and the slices by
\begin{equation*}
 D\defeq \{u:\exists w,\ (u,w)\in Y\},\qquad
 H_t\defeq\{w+tu:(u,w)\in Y\},\qquad s\defeq\dimH D.
\end{equation*}
The incidence set
\begin{equation*}
 \mathcal I(Y)\defeq \{(t,w+tu):t\in\Q_p,\ (u,w)\in Y\}\subset\Q_p^m
\end{equation*}
satisfies
\begin{equation*}
 \dimH\mathcal I(Y)\geq s+1.
\end{equation*}
Moreover,
\begin{equation*}
 \dimH H_t\geq s
 \quad\text{for Haar-almost every }t\in\Q_p.
\end{equation*}
For every $0\leq\sigma<s$,
\begin{equation}\label{eq:quantitative-line-slicing}
 \dimH\{t\in\Q_p:\dimH H_t\leq\sigma\}
 \leq1-\frac{s-\sigma}{m-1}.
\end{equation}
\end{theorem}
\begin{proof}
We first prove the incidence bound. For every $u\in D$, choose
$w$ such that $(u,w)\in Y$. The line
$\{(t,w+tu):t\in\Q_p\}$ lies in $\mathcal I(Y)$ and has direction
$[1:u]$. On bounded slope sets, $u\mapsto[1:u]$ is bi-Lipschitz onto its
image. Countable localization therefore gives
$\dimH\{[1:u]:u\in D\}=s$, and Theorem~\ref{thm:Qp-lines} implies
$\dimH\mathcal I(Y)\geq s+1$.
It remains to prove the assertions for individual slices.
We first work with an integral family and integral parameters,
and then remove these restrictions.

\emph{Slice measures in the integral model.}
Suppose that $Y\subset\Z_p^{m-1}\times\Z_p^{m-1}$ and
$t\in\Z_p$. If $s=0$, every $H_t$ is nonempty and the
conclusion is immediate. There are also no exceptional sets
to consider. We may therefore assume that $s>0$.
Fix $0<\tau<s$.
The set $D$ is analytic, being a coordinate projection of $Y$.
Lemma~\ref{lem:projection-frostman} supplies a probability measure $\nu$
carried by a compact subset of $D$ such that
\begin{equation*}
 \nu(B(u,\delta))\leq C_\nu\delta^\tau,
 \qquad \delta\defeq p^{-\ell},\quad\ell\geq1.
\end{equation*}
By Theorem~\ref{thm:projection-selection}, choose a Borel probability
measure $\lambda$ carried by $Y$ whose slope marginal is $\nu$.
Thus the slope distribution of $\lambda$ satisfies the preceding
Frostman bound. Define the slice measures and their averaged
densities by
\begin{equation*}
 \mu_t\defeq ((u,w)\mapsto w+tu)_*\lambda,
 \qquad
 G_\delta(t,z)\defeq \delta^{-(m-1)}\mu_t(B(z,\delta)).
\end{equation*}
For every $t$, the measure $\mu_t$ is a probability measure
carried by $H_t$. The density $G_\delta(t,\cdot)$ distributes
the mass of each residue class uniformly over that class.
We also average the slope measure at scale $\delta$ by setting
\begin{equation*}
 \nu_\delta(u)\defeq \delta^{-(m-1)}\nu(B(u,\delta)).
\end{equation*}
Every ball of radius $\delta$ in $\Z_p^{m-1}$ has Haar measure
$\delta^{m-1}$, so Tonelli's theorem gives
\begin{equation*}
 \int\nu_\delta(u)\,du
 =\delta^{-(m-1)}\int\!\int
       \one_{\{\|u-v\|\leq\delta\}}\,du\,d\nu(v)=1.
\end{equation*}
The Frostman estimate gives
$\|\nu_\delta\|_\infty\leq C_\nu\delta^{-(m-1-\tau)}$.
For $m'\defeq m/(m-1)$, the conjugate exponent to $m$, it follows that
\begin{align*}
 \|\nu_\delta\|_{L^{m'}}^{m'}
 &\leq\|\nu_\delta\|_\infty^{m'-1}\|\nu_\delta\|_1\notag\\
 &\leq C_\nu^{m'-1}
          \delta^{-(m-1-\tau)(m'-1)}.
\end{align*}
Since $(m'-1)/m'=1/m$, taking the $m'$th root yields
\begin{equation}\label{eq:projection-marginal-norm}
 \|\nu_\delta\|_{L^{m'}(\Z_p^{m-1})}
 \leq C\delta^{-(m-1-\tau)/m}.
\end{equation}

\emph{Duality and the averaged mass estimate.}
We now estimate the norm of $G_\delta$. Let $f\geq0$ belong
to $L^m(\Z_p^m)$. By the definition of $\mu_t$ and Tonelli's
theorem, we have
\begin{align*}
 \int_{\Z_p}\int_{\Z_p^{m-1}}f(t,z)G_\delta(t,z)\,dz\,dt
 &=\int_Y\int_{\Z_p}\delta^{-(m-1)}
       \int_{B(w+tu,\delta)}f(t,z)\,dz\,dt\,d\lambda(u,w)\notag\\
 &\leq\int_{\Z_p^{m-1}}\mathcal K_\delta f(u)\,d\nu(u).
\end{align*}
If $\|u-u'\|\leq\delta$, then
$\|tu-tu'\|\leq\delta$ for every $t\in\Z_p$. Equal-radius balls
with these centers coincide. Hence $\mathcal K_\delta f$ is constant on every slope ball
of radius $\delta$. It follows from the definition of
$\nu_\delta$ that
\begin{equation*}
 \int\mathcal K_\delta f\,d\nu
 =\int\mathcal K_\delta f(u)\nu_\delta(u)\,du.
\end{equation*}
H\"older's inequality, Lemma~\ref{lem:projection-tube-maximal}, and
\eqref{eq:projection-marginal-norm} now give
\begin{align*}
 \int fG_\delta
 &\leq\|\mathcal K_\delta f\|_{L^m}\|\nu_\delta\|_{L^{m'}}\notag\\
 &\leq C(\ell+1)^2\delta^{-(m-1-\tau)/m}\|f\|_{L^m}.
\end{align*}
Taking the supremum over all nonnegative $f$ with $L^m$ norm
at most one, and applying $L^m$--$L^{m'}$ duality, we obtain
\begin{equation}\label{eq:projection-dual-density}
 \|G_\delta\|_{L^{m'}(\Z_p^m)}
 \leq C(\ell+1)^2\delta^{-(m-1-\tau)/m}.
\end{equation}
All constants here may depend on $\nu$, but they are independent of
$\ell$.

To use this norm bound for the slice measures, define
\begin{equation}\label{eq:projection-fractional-mass}
 F_\delta(t)\defeq \int\mu_t(B(z,\delta))^{1/(m-1)}\,d\mu_t(z).
\end{equation}
Let $\mathcal B_\delta$ be the finite partition of $\Z_p^{m-1}$ into
balls of radius $\delta$, and write $m_B(t)\defeq \mu_t(B)$.
For $z\in B$, the ball $B(z,\delta)$ is exactly $B$, so
\begin{equation*}
 F_\delta(t)
 =\sum_{B\in\mathcal B_\delta}m_B(t)^{1+1/(m-1)}
 =\sum_{B\in\mathcal B_\delta}m_B(t)^{m'}.
\end{equation*}
This is the order-$m'$ moment of the masses of the partition
cells. Related projection results for dimension spectra were
studied by Hunt and Kaloshin \cite{HK}. In the present argument,
we can estimate this moment directly from the preceding norm bound.
Indeed,
\begin{align*}
 \int_{\Z_p^{m-1}}G_\delta(t,z)^{m'}\,dz
 &=\sum_{B\in\mathcal B_\delta}
       \delta^{m-1}\bigl(\delta^{-(m-1)}m_B(t)\bigr)^{m'}\notag\\
 &=\delta^{(m-1)(1-m')}\sum_{B\in\mathcal B_\delta}m_B(t)^{m'}
 =\delta^{-1}F_\delta(t).
\end{align*}
Thus \eqref{eq:projection-dual-density} implies
\begin{equation}\label{eq:projection-averaged-mass}
\begin{aligned}
 \int_{\Z_p}F_\delta(t)\,dt
 &=\delta\|G_\delta\|_{L^{m'}}^{m'}\\
 &\leq C(\ell+1)^{2m'}
       \delta^{\,1-(m-1-\tau)m'/m}\\
 &=C(\ell+1)^{2m'}\delta^{\tau/(m-1)}.
\end{aligned}
\end{equation}

\emph{Covers of the exceptional parameters.}
Fix $0\leq\sigma<\beta<\tau$, and put $\delta_\ell\defeq p^{-\ell}$.
Let the set of parameters where the moment exceeds the desired
bound be
\begin{equation*}
 A_\ell(\beta)\defeq \{t\in\Z_p:
       F_{\delta_\ell}(t)>\delta_\ell^{\beta/(m-1)}\}.
\end{equation*}
By Markov's inequality and \eqref{eq:projection-averaged-mass},
\begin{align*}
 \mathcal L^1(A_\ell(\beta))
 &\leq\delta_\ell^{-\beta/(m-1)}
                        \int_{\Z_p}F_{\delta_\ell}(t)\,dt\notag\\
 &\leq C(\ell+1)^{2m'}
                  \delta_\ell^{(\tau-\beta)/(m-1)}.
\end{align*}
We claim that $A_\ell(\beta)$ is a union of parameter residue
classes of radius $\delta_\ell$. To see this, suppose that
$|t-t'|_p\leq\delta_\ell$. Then
\begin{equation*}
 \|(w+tu)-(w+t'u)\|
 =|t-t'|_p\|u\|\leq\delta_\ell
 \qquad ((u,w)\in Y).
\end{equation*}
Therefore the two evaluations lie in the same spatial residue class,
$\mu_t(B)=\mu_{t'}(B)$ for every $B\in\mathcal B_{\delta_\ell}$,
and $F_{\delta_\ell}(t)=F_{\delta_\ell}(t')$.
This proves the claim. Each parameter ball has Haar measure
$\delta_\ell$, so the measure estimate above shows that the
number of balls in $A_\ell(\beta)$ is at most
\begin{equation}\label{eq:projection-bad-cover}
 C(\ell+1)^{2m'}
       \delta_\ell^{-1+(\tau-\beta)/(m-1)}.
\end{equation}

Let
\begin{equation*}
 A(\beta)\defeq \limsup_{\ell\to\infty}A_\ell(\beta)
 =\bigcap_{\ell_0\geq1}\bigcup_{\ell\geq\ell_0}A_\ell(\beta).
\end{equation*}
For $\eta>0$, set $\alpha\defeq 1-(\tau-\beta)/(m-1)+\eta$.
The covers in \eqref{eq:projection-bad-cover}, with $\ell\geq\ell_0$,
cover $A(\beta)$ and have total $\alpha$-dimensional cost at most
\begin{align*}
 C\sum_{\ell\geq\ell_0}(\ell+1)^{2m'}
       \delta_\ell^{-1+(\tau-\beta)/(m-1)}\delta_\ell^\alpha
 &=C\sum_{\ell\geq\ell_0}(\ell+1)^{2m'}p^{-\ell\eta}.
\end{align*}
The series converges. As $\ell_0\to\infty$, both the total
cost and the largest covering radius tend to zero. Therefore
$\HH^\alpha(A(\beta))=0$. Letting $\eta\downarrow0$, we obtain
\begin{equation}\label{eq:projection-limsup-bound}
 \dimH A(\beta)\leq1-\frac{\tau-\beta}{m-1}.
\end{equation}

\emph{From mass decay to slice dimension.}
We now show that $\dimH H_t\geq\beta$ whenever
$t\notin A(\beta)$. Fix such a parameter. By the definition of
the limsup set, for every sufficiently large $\ell$ we have
\begin{equation*}
 F_{\delta_\ell}(t)\leq\delta_\ell^{\beta/(m-1)}.
\end{equation*}
Fix $0<\gamma<\beta$, and define
\begin{equation*}
 E_\ell\defeq \{z:\mu_t(B(z,\delta_\ell))>\delta_\ell^\gamma\}.
\end{equation*}
On $E_\ell$, the integrand in \eqref{eq:projection-fractional-mass}
is larger than $\delta_\ell^{\gamma/(m-1)}$. Hence
\begin{align*}
 \mu_t(E_\ell)
 &\leq\delta_\ell^{-\gamma/(m-1)}F_{\delta_\ell}(t)\notag\\
 &\leq\delta_\ell^{(\beta-\gamma)/(m-1)}.
\end{align*}
Since the right-hand side is summable in $\ell$, Borel--Cantelli
implies that $\mu_t$-almost every $z$ belongs to only finitely
many sets $E_\ell$. Thus
\begin{equation*}
 \mu_t(B(z,p^{-\ell}))\leq p^{-\ell\gamma}
 \quad\text{for all sufficiently large $\ell$, for $\mu_t$-almost every $z$}.
\end{equation*}

To obtain the Hausdorff dimension bound, we restrict to a set
where this estimate holds uniformly. For $J\geq1$, define
\begin{equation*}
 K_J\defeq \bigcap_{\ell\geq J}
       \{z\in\Z_p^{m-1}:\mu_t(B(z,p^{-\ell}))\leq p^{-\ell\gamma}\}.
\end{equation*}
These sets are Borel, and $\mu_t(\bigcup_JK_J)=1$.
Choose $J$ with $\mu_t(K_J)>0$, and define a Borel probability measure
\begin{equation*}
 \widetilde\mu(A)\defeq \frac{\mu_t(A\cap K_J)}{\mu_t(K_J)}.
\end{equation*}
It is still carried by $H_t$. If a ball $B(x,p^{-\ell})$, $\ell\geq J$,
meets $K_J$, choose $z$ in the intersection. The ultrametric property
gives $B(x,p^{-\ell})=B(z,p^{-\ell})$, and therefore
\begin{equation*}
 \widetilde\mu(B(x,p^{-\ell}))
 \leq\frac{p^{-\ell\gamma}}{\mu_t(K_J)}.
\end{equation*}
The same bound is trivial when the ball misses $K_J$. For an arbitrary
small radius $\rho$, choose $\ell\geq J$ with
$p^{-\ell}\leq\rho<p^{-\ell+1}$. The discrete set of distances gives
$B(x,\rho)=B(x,p^{-\ell})$, and hence
\begin{equation*}
 \widetilde\mu(B(x,\rho))
 \leq\frac{p^{-\ell\gamma}}{\mu_t(K_J)}
 \leq\frac{\rho^\gamma}{\mu_t(K_J)}.
\end{equation*}
We have therefore obtained a Frostman bound for every sufficiently
small radius. If $\{B_i\}$ is any small-ball cover of $H_t$, then
\begin{equation*}
 1\leq\sum_i\widetilde\mu(B_i)
 \leq\frac1{\mu_t(K_J)}\sum_i(\operatorname{rad}B_i)^\gamma.
\end{equation*}
It follows that $\HH^\gamma(H_t)>0$ and hence
$\dimH H_t\geq\gamma$. Letting $\gamma\uparrow\beta$
proves $\dimH H_t\geq\beta$, as claimed. Consequently,
\begin{equation*}
 \{t\in\Z_p:\dimH H_t\leq\sigma\}\subset A(\beta),
\end{equation*}
because $\sigma<\beta$. By \eqref{eq:projection-limsup-bound},
\begin{equation*}
 \dimH\{t\in\Z_p:\dimH H_t\leq\sigma\}
 \leq1-\frac{\tau-\beta}{m-1}.
\end{equation*}
Letting first $\beta\downarrow\sigma$ and then $\tau\uparrow s$,
we obtain \eqref{eq:quantitative-line-slicing} for the integral
family with parameters in $\Z_p$.

\emph{Localization and the almost-everywhere lower bound.}
It remains to treat an arbitrary analytic family and all
parameters in $\Q_p$. For each integer $A_0\geq0$, define
\begin{equation*}
 Y_{A_0}\defeq Y\cap\{(u,w):\|u\|,\|w\|\leq p^{A_0}\},
 \qquad D_{A_0}\defeq \{u:\exists w,\ (u,w)\in Y_{A_0}\}.
\end{equation*}
Both $Y_{A_0}$ and $D_{A_0}$ are analytic. Since $D=\bigcup_{A_0\geq0}D_{A_0}$,
countable stability gives $s=\sup_{A_0}\dimH D_{A_0}$. Given
$0\leq\sigma<\tau<s$, choose $A_0$ with $s_{A_0}\defeq \dimH D_{A_0}>\tau$.
The slices of $Y_{A_0}$ are contained in the corresponding
slices of $Y$. Thus, if an original slice has dimension at most
$\sigma$, the same is true for its $Y_{A_0}$-slice. It is
therefore enough to estimate the exceptional parameters of
$Y_{A_0}$.

Fix a parameter ball $p^{-A}\Z_p$, where $A\geq0$ is an integer, and write
$t_0\defeq p^At\in\Z_p$, so that $t=p^{-A}t_0$.
Choose an integer $\kappa\geq A+A_0$, and
make the invertible linear change
\begin{equation*}
 u'\defeq p^{\kappa-A}u,\qquad w'\defeq p^{\kappa}w.
\end{equation*}
For $(u,w)\in Y_{A_0}$,
\begin{equation*}
 \|u'\|\leq p^{-\kappa+A+A_0}\leq1,\qquad
 \|w'\|\leq p^{-\kappa+A_0}\leq1.
\end{equation*}
The transformed family is integral, and the invertible linear
change preserves the slope dimension $s_{A_0}$. At parameter
$t_0$, its slice is
\begin{equation*}
 \{w'+t_0u':(u,w)\in Y_{A_0}\}
 =p^{\kappa} H^{Y_{A_0}}_{p^{-A}t_0},
\end{equation*}
where $H^{Y_{A_0}}$ denotes the slices formed using only $Y_{A_0}$.
Applying the integral result, and using that scalar similarities
preserve Hausdorff dimension, gives
\begin{align*}
 \dimH\{t\in p^{-A}\Z_p:\dimH H_t\leq\sigma\}
 &\leq1-\frac{s_{A_0}-\sigma}{m-1}\notag\\
 &\leq1-\frac{\tau-\sigma}{m-1}.
\end{align*}
Now let $\tau\uparrow s$ and use
$\Q_p=\bigcup_{A\geq0}p^{-A}\Z_p$ to obtain the global bound.

Finally, if $s>0$, choose $0\leq\sigma_i<s$ with $\sigma_i\uparrow s$.
The set of parameters with $\dimH H_t<s$ is the countable union of
$\{t:\dimH H_t\leq\sigma_i\}$. Each has dimension strictly less than
one by \eqref{eq:quantitative-line-slicing}, and hence Haar outer
measure zero: on $\Q_p$, a ball of radius $p^{-\ell}$ has Haar
measure $p^{-\ell}$. Their countable union is Haar-null, which proves the
almost-everywhere assertion. The case $s=0$ is immediate because
the slices are nonempty. This completes the proof.
\end{proof}

\section{Polynomial evaluation images}
\label{sec:polynomial-projections}

In this section, $P_r$ is the family \eqref{eq:projection-family},
with its fixed output dimension, degree, and scalar weights. We first prove the total-incidence
bound using the curved Kakeya theorem. We then prove
Theorem~\ref{thm:polynomial-projections} by applying the
projection--incidence theorem to a lifted line family. Finally,
we translate the parameter to obtain the arbitrary-reference
comparison and the almost-everywhere maximal dimension statement.

\subsection*{Total incidence and the reference image}
For $\mathcal E\subset(\Q_p^n)^{d+1}$ and $U\subset\Q_p$, define
the polynomial incidence set over $U$ by
\begin{equation*}
 \mathcal I_U(\mathcal E)
 \defeq \{(r,P_r(v)):r\in U,\ v\in\mathcal E\}.
\end{equation*}
For $r\in U$, its slice is $P_r(\mathcal E)$. In particular, write
\begin{equation*}
 \mathcal I^\times(\mathcal E)
 \defeq \mathcal I_{\Q_p^\times}(\mathcal E)
 =\{(r,y):r\in\Q_p^\times,\ y\in P_r(\mathcal E)\}.
\end{equation*}
We give a lower bound for this incidence set in terms of one
reference image. The upper bound follows from the dimension of
the coefficient set. For the lower bound, we reverse the polynomial
and apply Theorem~\ref{thm:padic}.

\begin{proposition}[Incidence dimension from one evaluation]
\label{prop:polynomial-total-incidence}
Fix the family \eqref{eq:projection-family}, and let
$\mathcal E\subset(\Q_p^n)^{d+1}$ be any nonempty set. Then
\begin{equation*}
 \dimH P_0(\mathcal E)+1
 \leq\dimH\mathcal I^\times(\mathcal E)
 \leq\min\{\dimH\mathcal E+1,n+1\}.
\end{equation*}
In particular, for every $0\leq S\leq n$,
\begin{equation}\label{eq:polynomial-incidence-threshold}
 \dimH P_0(\mathcal E)\geq S
 \quad\Longrightarrow\quad
 \dimH\mathcal I^\times(\mathcal E)\geq S+1.
\end{equation}
\end{proposition}
\begin{proof}
We begin with the lower bound. Let
$S_0\defeq \dimH P_0(\mathcal E)$. For each
$v=(v_0,\ldots,v_d)\in\mathcal E$, define
\begin{equation*}
 \Gamma_v(t)
 \defeq \left(\sum_{j=0}^d a_jv_jt^{d+1-j},\ t^{d+1}\right)
 \in\Q_p^{n+1},\qquad t\in\Q_p.
\end{equation*}
The last coordinate is $t^{d+1}$, so the map has degree exactly
$d+1$. Its leading vector is $(a_0v_0,1)$ and is nonzero for
every tuple. Thus the set of leading directions is
\begin{equation*}
 V_0\defeq \{[x:1]:x\in P_0(\mathcal E)\}\subset\PP^n(\Q_p),
 \qquad\dimH V_0=S_0.
\end{equation*}

Let $B\defeq \bigcup_{v\in\mathcal E}\Gamma_v(\Q_p)$.
For every direction in $V_0$, the set $B$ contains the image
$\Gamma_v(\Z_p)$ of a corresponding polynomial. By
Theorem~\ref{thm:padic}, applied in dimension $n+1$ and degree
$d+1$, we have
\begin{equation}\label{eq:augmented-incidence-lower}
 \dimH B\geq S_0+1.
\end{equation}

For $t\neq0$, we see that
\begin{equation*}
 \Gamma_v(t)=\bigl(t^{d+1}P_{1/t}(v),t^{d+1}\bigr),
 \qquad\Gamma_v(0)=0.
\end{equation*}
Consequently, for
\begin{equation*}
 \Theta:\Q_p^\times\times\Q_p^n\longrightarrow\Q_p^{n+1},
 \qquad\Theta(r,y)\defeq (r^{-(d+1)}y,r^{-(d+1)}),
\end{equation*}
we have the exact set identity
\begin{equation}\label{eq:augmented-incidence-image}
 B=\{0\}\cup\Theta\bigl(\mathcal I^\times(\mathcal E)\bigr).
\end{equation}
On each set where
$|r|_p$ is fixed and $\|y\|$ is bounded, the map $\Theta$ is
Lipschitz. These sets give a countable cover of the domain. Therefore
\eqref{eq:augmented-incidence-image} gives
$\dimH B\leq\dimH\mathcal I^\times(\mathcal E)$.
Combining this with \eqref{eq:augmented-incidence-lower} proves
the lower bound.

For the upper bound, the map
\begin{equation*}
 F:\mathcal E\times\Q_p^\times\longrightarrow\mathcal I^\times(\mathcal E),
 \qquad F(v,r)\defeq (r,P_r(v))
\end{equation*}
is surjective and Lipschitz on every bounded box. The full-factor lemma \eqref{eq:product-dimension} therefore give
\begin{equation*}
 \dimH\mathcal I^\times(\mathcal E)
 \leq\dimH(\mathcal E\times\Q_p^\times)
 =\dimH\mathcal E+1.
\end{equation*}
The incidence set is contained in $\Q_p^{n+1}$, so its dimension
is also at most $n+1$. Combining these two upper bounds completes
the proof.
\end{proof}

\begin{remark}[Coefficient-set dimension does not suffice]
\label{rem:coefficient-dimension-counterexample}
The condition on the reference image in
\eqref{eq:polynomial-incidence-threshold} cannot be replaced by
$\dimH\mathcal E\geq S$. This fails even when $\mathcal E$
is compact, all scalar coefficients are nonzero, and $S\leq n$.
To see this, take $n=2$, $d=1$, and $a_0=a_1=1$.
Let $e\defeq (1,0)\in\Q_p^2$, and define
\begin{equation*}
 \mathcal E\defeq \{(xe,ye):x,y\in\Z_p\}\subset(\Q_p^2)^2.
\end{equation*}
The map $(x,y)\mapsto(xe,ye)$ is an isometry, so
$\dimH\mathcal E=2$. If $\mathcal L\defeq\Q_pe$, then
\begin{equation*}
 P_r(\mathcal E)=(\Z_p+r\Z_p)e\subset\mathcal L
 \quad\text{for every }r\in\Q_p.
\end{equation*}
In fact, each of these images has dimension one, because it contains
$\Z_pe$ and is contained in $\mathcal L$. Moreover,
\begin{equation*}
 \Z_p^\times\times(\Z_pe)
 \subset\mathcal I^\times(\mathcal E)
 \subset\Q_p^\times\times\mathcal L.
\end{equation*}
Both bounding products have dimension two. Thus, with $S\defeq 2=n$,
\begin{equation*}
 \dimH\mathcal E=S=2,
 \qquad\dimH\mathcal I^\times(\mathcal E)=2<3=S+1.
\end{equation*}
Here $\dimH P_0(\mathcal E)=1$, so the lower bound in
Proposition~\ref{prop:polynomial-total-incidence} is attained. Here $S+1=3$ is the dimension of the ambient space
$\Q_p^{n+1}$, so the failure is not due to the ambient dimension bound.
It occurs because the two coefficient blocks vary independently
in one spatial line, while every evaluation still belongs to that
line.
\end{remark}

\subsection*{Polynomial reversal and individual slices}
We next prove Theorem~\ref{thm:polynomial-projections}.
The incidence bound alone does not give the dimension of individual
evaluation images. Instead, we lift the reversed polynomials to an
analytic family of lines. We shall compute both its slope set and
its slices, and then apply
Theorem~\ref{thm:quantitative-line-slicing}. 

\begin{proof}[Proof of Theorem~\ref{thm:polynomial-projections}]
If $S=0$, the lower bound is immediate because the evaluation
images are nonempty. There is also no $\sigma$ in the exceptional
range. We may therefore assume that $S>0$.

\emph{Reversal and the lifted slope set.}
We first include the fixed scalar weights in the coefficient
vectors. Define
\begin{equation*}
 \mathcal A\defeq \{(b_0,\ldots,b_d):
              b_j=a_jv_j,\ (v_0,\ldots,v_d)\in\mathcal E\}.
\end{equation*}
This is analytic because it is a continuous linear image of
$\mathcal E$. For $b\in\mathcal A$, write
$f_b(r)\defeq \sum_{j=0}^db_jr^j$. Then
\begin{equation*}
 \{f_b(r):b\in\mathcal A\}=P_r(\mathcal E),\qquad
 \{b_0:b\in\mathcal A\}=P_0(\mathcal E).
\end{equation*}
We reverse this polynomial as before so its leading coefficient is
the reference value $b_0$. Set
\begin{equation*}
 q_b(t)\defeq \sum_{j=0}^db_jt^{d-j}
 =t^df_b(1/t),\qquad t\neq0.
\end{equation*}

Choose free vectors $c_1,\ldots,c_{d-1}\in\Q_p^n$, and set
$c_0\defeq b_0$ and $c_d\defeq 0 \in \Q_p^n$. Define a slope and an intercept in $\Q_p^{dn}$ by
\begin{equation}\label{eq:projection-lifted-line}
 u\defeq (c_0,\ldots,c_{d-1}),\qquad
 w\defeq (b_1-c_1,\ldots,b_d-c_d).
\end{equation}
Let $Y$ be the set of all pairs $(u,w)$ arising from
$b\in\mathcal A$ and these free vectors. The set $Y$ is a continuous
image of $\mathcal A\times(\Q_p^n)^{d-1}$ and is therefore analytic.
It parametrizes affine lines
\begin{equation*}
 L_{b,c}(t)\defeq (t,w+tu)\in\Q_p^{dn+1}.
\end{equation*}
Observe that the first slope block is $b_0$, while the remaining
blocks are arbitrary. Every $b_0\in P_0(\mathcal E)$ is the
first coordinate of some $b\in\mathcal A$. Consequently, the
slope set is exactly
\begin{equation*}
 D=P_0(\mathcal E)\times(\Q_p^n)^{d-1},
 \qquad \dimH D=S+(d-1)n,
\end{equation*}
where the dimension identity is Lemma~\ref{lem:products}.
When $d=1$, there are no free blocks and $D=P_0(\mathcal E)$.

\emph{The full inverse-image identity.}
We now compute the slices
$H_t\defeq \{w+tu:(u,w)\in Y\}$. For a vector
$z=(z_0,\ldots,z_{d-1})\in(\Q_p^n)^d$, put
\begin{equation*}
 R_t(z)\defeq \sum_{j=0}^{d-1}t^{d-1-j}z_j.
\end{equation*}
The coefficient of $z_{d-1}$ is $1$, so $R_t$ is surjective for every
$t$. Moreover,
\begin{equation*}
 R_t(z)=\Phi(z_{d-1},z_{d-2},\ldots,z_0,t).
\end{equation*}
Thus $R_t$ is the slice map of the polynomial lift, with the
order of the vector blocks reversed.

By \eqref{eq:projection-lifted-line}, the $j$th block of $w+tu$
is $b_{j+1}-c_{j+1}+tc_j$. Substitution gives
\begin{equation*}
\begin{aligned}
 R_t(w+tu)
 &=\sum_{j=0}^{d-1}t^{d-1-j}
                  (b_{j+1}-c_{j+1}+tc_j)\\
 &=\sum_{i=1}^{d}t^{d-i}b_i
   -\sum_{i=1}^{d}t^{d-i}c_i
   +\sum_{i=0}^{d-1}t^{d-i}c_i\\
 &=\sum_{i=1}^{d}t^{d-i}b_i+t^dc_0-c_d\\
 &=b_0t^d+b_1t^{d-1}+\cdots+b_d=q_b(t).
\end{aligned}
\end{equation*}
The terms involving $c_1,\ldots,c_{d-1}$ cancel. Hence, for
fixed $b$, all the resulting slice points belong to
$R_t^{-1}(q_b(t))$.

It remains to prove the reverse inclusion. Let $z$ satisfy
$R_t(z)=q_b(t)$. Starting with $c_0\defeq b_0$, define
successively
\begin{equation*}
 c_{j+1}\defeq b_{j+1}+tc_j-z_j,\qquad 0\leq j\leq d-1.
\end{equation*}
Iteration gives
\begin{equation*}
 c_d=\sum_{i=0}^{d}t^{d-i}b_i
       -\sum_{j=0}^{d-1}t^{d-1-j}z_j
     =q_b(t)-R_t(z)=0.
\end{equation*}
Thus the final value satisfies $c_d=0$, as required. The
preceding values define the free vectors $c_1,\ldots,c_{d-1}$,
and the block equations give $z=w+tu$. When $d=1$, the same
argument is the single equation $z_0=b_1+tb_0$.
Taking the union over $b\in\mathcal A$, we conclude that
\begin{equation}\label{eq:projection-entire-fibers}
 H_t=R_t^{-1}\bigl(t^dP_{1/t}(\mathcal E)\bigr),\qquad t\neq0.
\end{equation}

To calculate its dimension, define the linear coordinate map
\begin{equation*}
 J_t(z_0,\ldots,z_{d-1})
   \defeq \bigl(R_t(z),z_0,\ldots,z_{d-2}\bigr).
\end{equation*}
Its inverse is
\begin{equation*}
 J_t^{-1}(y,z_0,\ldots,z_{d-2})
 =\left(z_0,\ldots,z_{d-2},
           y-\sum_{j=0}^{d-2}t^{d-1-j}z_j\right).
\end{equation*}
Empty sums and lists are omitted when $d=1$.
For every fixed $t$, this map and its inverse are Lipschitz. By \eqref{eq:projection-entire-fibers},
we have
\begin{equation*}
 J_t(H_t)=t^dP_{1/t}(\mathcal E)\times(\Q_p^n)^{d-1}.
\end{equation*}
Lemma~\ref{lem:products} and the fact that multiplication by $t^d\neq0$
is a similarity give the exact dimension formula
\begin{equation}\label{eq:projection-fiber-dimension}
 \dimH H_t=(d-1)n+\dimH P_{1/t}(\mathcal E),\qquad t\neq0.
\end{equation}
This is the required dimension identity for each nonzero-parameter
slice.

\emph{Cancellation of the auxiliary dimensions.}
Apply Theorem~\ref{thm:quantitative-line-slicing} with $m=dn+1$.
The family $Y$ is nonempty and analytic, and its slope dimension is
$S+(d-1)n\leq dn$. By \eqref{eq:projection-fiber-dimension},
for Haar-almost every $t\neq0$ we have
\begin{equation*}
 (d-1)n+\dimH P_{1/t}(\mathcal E)
 =\dimH H_t\geq S+(d-1)n.
\end{equation*}
Subtracting $(d-1)n$ gives $\dimH P_{1/t}(\mathcal E)\geq S$.
Inversion is a similarity on each valuation annulus in
$\Q_p^\times$, so it preserves Hausdorff dimension and Haar-null
sets. This proves \eqref{eq:projection-ae}.

For $0\leq\sigma<S$, the exact slice identity gives
\begin{equation*}
 \{t\neq0:\dimH P_{1/t}(\mathcal E)\leq\sigma\}
 =\{t\neq0:\dimH H_t\leq\sigma+(d-1)n\}.
\end{equation*}
Since $\sigma+(d-1)n<S+(d-1)n$, the exceptional-set estimate
in the same theorem yields
\begin{align*}
 \dimH\{t\neq0:\dimH P_{1/t}(\mathcal E)\leq\sigma\}
 &\leq1-\frac{[S+(d-1)n]-[\sigma+(d-1)n]}{dn}\notag\\
 &=1-\frac{S-\sigma}{dn}.
\end{align*}
Inversion transfers this bound to $r\neq0$. The parameter $r=0$
is not exceptional because $\sigma<S$. This proves
\eqref{eq:projection-exceptional} and completes the proof.
\end{proof}

\subsection*{Reference parameters and maximal dimension}
We now replace the reference image $P_0(\mathcal E)$ by an
arbitrary evaluation image. Expanding the polynomial about the
chosen parameter gives a new coefficient set, obtained by a linear
map from the original one. The preceding theorem then applies
to this new family.

\begin{proposition}[An arbitrary reference parameter]
\label{prop:arbitrary-reference-parameter}
For the family \eqref{eq:projection-family}, let
$\mathcal E\subset(\Q_p^n)^{d+1}$ be nonempty and analytic, and fix
$r_0\in\Q_p$. Put $S_{r_0}\defeq \dimH P_{r_0}(\mathcal E)$. Then
\begin{equation*}
 \dimH P_r(\mathcal E)\geq S_{r_0}
 \quad\text{for Haar-almost every }r\in\Q_p.
\end{equation*}
For every $0\leq\sigma<S_{r_0}$,
\begin{equation*}
 \dimH\{r\in\Q_p:\dimH P_r(\mathcal E)\leq\sigma\}
 \leq1-\frac{S_{r_0}-\sigma}{dn}.
\end{equation*}
\end{proposition}
\begin{proof}
Fix the reference parameter $r_0$. For $0\leq j\leq d$, define
\begin{equation*}
 b_j(v)\defeq\sum_{j'=j}^{d}\binom{j'}j r_0^{j'-j}a_{j'}v_{j'}.
\end{equation*}
The binomial theorem gives
\begin{equation*}
 P_{r_0+t}(v)=\sum_{j=0}^{d}t^jb_j(v),\qquad
 b_0(v)=P_{r_0}(v),\quad b_d(v)=a_dv_d.
\end{equation*}
Let $\mathcal E_{r_0}$ be the image of $\mathcal E$ under the
continuous linear map $v\mapsto(b_0(v),\ldots,b_d(v))$.
It is nonempty and analytic. For the unweighted degree-$d$ family
$Q_t(w_0,\ldots,w_d)=\sum_{j=0}^dt^jw_j$, we have
\begin{equation*}
 Q_t(\mathcal E_{r_0})=P_{r_0+t}(\mathcal E),\qquad
 Q_0(\mathcal E_{r_0})=P_{r_0}(\mathcal E).
\end{equation*}
All fixed weights of $Q_t$ are one, so its leading weight is
nonzero. Applying Theorem~\ref{thm:polynomial-projections} with
reference dimension $S_{r_0}$ gives the stated bounds with
denominator $dn$. Translation preserves Hausdorff dimension and
Haar-null sets, so these are precisely the required assertions.
\end{proof}

We next prove the maximal-dimension corollary. Although the
reference parameter is arbitrary, we only need a countable sequence
of references whose image dimensions approach the supremum.

\begin{proof}[Proof of Corollary~\ref{cor:maximal-projection-dimension}]
Since the images are nonempty subsets of $\Q_p^n$, we have
$0\leq S_*\leq n$. The case $S_*=0$ is immediate, so assume
that $S_*>0$. We first prove the almost-everywhere equality.
For each $j\geq1$, choose
$r_j\in\Q_p$ such that
\begin{equation*}
 S_*-\frac1j<S_j\defeq \dimH P_{r_j}(\mathcal E)\leq S_*.
\end{equation*}
In particular, $S_j\to S_*$ and $\sup_jS_j=S_*$.

For each fixed $j$, Proposition~\ref{prop:arbitrary-reference-parameter}
provides a Borel Haar-null set $Z_j$ outside which
$\dimH P_r(\mathcal E)\geq S_j$.
The union $Z\defeq \bigcup_{j\geq1}Z_j$ is still Haar-null. If $r\notin Z$,
all these inequalities hold simultaneously, and hence
\begin{equation*}
 \dimH P_r(\mathcal E)\geq\sup_{j\geq1}S_j=S_*.
\end{equation*}
The reverse inequality holds for every $r$, because $S_*$ is the
supremum over all evaluation dimensions. Thus
$\dimH P_r(\mathcal E)=S_*$ for Haar-almost every $r$, proving
\eqref{eq:projection-maximal-ae}. Only a countable union of null sets is used. In particular,
we did not assume that the supremum was attained, but the proof
shows that it is attained at almost every parameter.

It remains to prove the exceptional-set bound. Fix
$0\leq\sigma<S_*$, and let
\begin{equation*}
 E_\sigma\defeq \{r\in\Q_p:\dimH P_r(\mathcal E)\leq\sigma\}.
\end{equation*}
Since $S_j\to S_*>\sigma$, for all sufficiently large $j$ the
quantitative part of Proposition~\ref{prop:arbitrary-reference-parameter}
applies to this same set and gives
\begin{equation*}
 \dimH E_\sigma\leq1-\frac{S_j-\sigma}{dn}.
\end{equation*}
The left-hand side is independent of $j$. Taking the limit of the
right-hand side as $j\to\infty$ gives
\begin{equation*}
 \dimH E_\sigma\leq1-\frac{S_*-\sigma}{dn},
\end{equation*}
as required by \eqref{eq:projection-maximal-exceptional}.
This completes the proof.
\end{proof}

\begin{remark}[Dimension-preserving evaluations]
\label{rem:dimension-preserving-evaluations}
In Remark~\ref{rem:coefficient-dimension-counterexample}, we
constructed a compact set with $S_*=1$ and
$\dimH\mathcal E=2$. Thus the maximal evaluation dimension may
be smaller than the coefficient-set dimension. Suppose instead
that $S\defeq\dimH\mathcal E$ and $\dimH P_0(\mathcal E)=S$.
The incidence bounds in
Proposition~\ref{prop:polynomial-total-incidence} then become
\begin{equation*}
 S+1\leq\dimH\mathcal I^\times(\mathcal E)
       \leq\dimH\mathcal E+1=S+1,
\end{equation*}
so $\dimH\mathcal I^\times(\mathcal E)=S+1$.

Every fixed evaluation $P_r$ is a linear Lipschitz map. Consequently,
\begin{equation*}
 S=\dimH P_0(\mathcal E)\leq S_*
   \leq\dimH\mathcal E=S.
\end{equation*}
Thus $S_*=S$. If $\mathcal E$ is analytic,
Corollary~\ref{cor:maximal-projection-dimension} then gives
$\dimH P_r(\mathcal E)=S$ for Haar-almost every $r$.
In fact, when $\mathcal{E}$ is analytic, preservation of coefficient-set dimension at any evaluation forces its preservation by almost every evaluation.
\end{remark}

\begin{remark}[Sharpness and the scope of the projection theorem]
\label{rem:projection-sharpness}
Remark~\ref{rem:coefficient-dimension-counterexample} shows why
$S$ cannot in general be replaced by
$\min\{n,\dimH\mathcal E\}$ in the vector-valued setting. The almost-everywhere assertion cannot be replaced by an
assertion for every parameter. Take $d=1$ and $a_0=a_1=1$.
Choose $r_*\in\Q_p^\times$, and let
$B\subset\Q_p^n$ be compact with $S\defeq \dimH B>0$. The compact set
\begin{equation*}
 \mathcal E_{r_*}\defeq \{(-r_*x,x):x\in B\}
\end{equation*}
satisfies
\begin{equation*}
 P_r(\mathcal E_{r_*})=(r-r_*)B,
 \qquad
 \dimH P_r(\mathcal E_{r_*})=
 \begin{cases}
  0,&r=r_*,\\
  S,&r\neq r_*.
 \end{cases}
\end{equation*}
The reference image has dimension $S$, since $r_*\neq0$, but the
image collapses at $r_*$. For every $0\leq\sigma<S$, the exceptional
set is exactly $\{r_*\}$. 

However, we do not claim here that the bound on the
exceptional set is optimal. In some scalar cases, stronger bounds are known. For the scalar linear family
$P_r(x,y)\defeq x+ry$ and a nonempty Borel set
$\mathcal E\subset\Q_p^2$, put
$T\defeq\dimH\mathcal E$ and $S\defeq\dimH P_0(\mathcal E)$. If
$0\leq\sigma<S$, Shen's theorem \cite[Theorem~1.4]{Shen}, applied
with a threshold $u$ satisfying $\sigma<u<S\leq\min\{T,1\}$,
gives
\begin{equation*}
 \dimH\{r:\dimH P_r(\mathcal E)\leq\sigma\}
 \leq\max\{2u-T,0\}.
\end{equation*}
Here the set on the left is contained in the set where
$\dimH P_r(\mathcal E)<u$. Letting $u\downarrow\sigma$ gives
$\max\{2\sigma-T,0\}$, which is often strictly smaller than the
bound $1-S+\sigma$ given by the present theorem.
\end{remark}

\medskip
\section{Block-valued polynomial evaluations}
\label{sec:block-evaluations}

We use the block family \eqref{eq:block-evaluation-family}, with
its fixed degrees $d_i$, output dimensions $n_i$, and scalar weights.
Recall that $N=\sum_i n_i$ and that $\Lambda$ is given by
\eqref{eq:block-denominator}. For the lifts in this section and
Section~\ref{sec:anisotropic-covering}, put
\begin{equation}\label{eq:block-lifting-dimensions}
 h_i\defeq\max\{1,d_i\},\qquad
 \Lambda=\sum_{i=1}^k n_i h_i,\qquad
 L\defeq\Lambda-N=\sum_{i=1}^k n_i(h_i-1).
\end{equation}
Thus $\Lambda+1$ is the lifted ambient dimension and $L$ is the
number of free auxiliary coordinates. Only a constant block needs
padding from degree $0$ to degree bound $h_i=1$.

For total incidence, we group coefficients of equal powers and
apply Proposition~\ref{prop:polynomial-total-incidence}. For individual
images, we lift each block separately to retain the denominator
$\Lambda$. Both arguments keep the full joint coefficient set.

\subsection*{Joint coefficients and total incidence}
For any nonempty $\mathcal E\subset\mathcal V$, absorb the fixed
scalar weights into the coefficient vectors by setting
\begin{equation}\label{eq:block-absorbed-coefficients}
 b_{i,j}(v)\defeq a_{i,j}v_{i,j},\qquad 0\leq j\leq d_i.
\end{equation}
When $d_i=0$, adjoin $b_{i,1}=0$. The resulting tuples, with
$0\leq j\leq h_i$, form the joint coefficient set $\mathcal A$.
The padding changes neither the original degree $d_i$ nor any
evaluation image. The vectors $b_{i,j}$ may vanish, including the
leading ones. The map $v\mapsto b(v)$ is continuous and linear,
so $\mathcal A$ is analytic whenever $\mathcal E$ is analytic. Let
\begin{equation*}
 f_{i,b}(r)\defeq\sum_{j=0}^{h_i}b_{i,j}r^j,\qquad
 f_b(r)\defeq(f_{i,b}(r))_{i=1}^k,
\end{equation*}
we have
\begin{equation*}
 \{f_b(r):b\in\mathcal A\}=\widetilde P_r(\mathcal E),\qquad
 \{(b_{i,0})_{i=1}^k:b\in\mathcal A\}
       =\widetilde P_0(\mathcal E).
\end{equation*}

For $U\subset\Q_p$, define
\begin{equation*}
 \widetilde{\mathcal I}_U(\mathcal E)
 \defeq\{(r,\widetilde P_r(v)):r\in U,\ v\in\mathcal E\}.
\end{equation*}
Write $\widetilde{\mathcal I}(\mathcal E)$ for
$\widetilde{\mathcal I}_{\Q_p}(\mathcal E)$ and
$\widetilde{\mathcal I}^{\times}(\mathcal E)$ for
$\widetilde{\mathcal I}_{\Q_p^\times}(\mathcal E)$.

\begin{proposition}[Block incidence dimension from reference evaluations]
\label{prop:block-total-incidence}
For the family \eqref{eq:block-evaluation-family}, let
$\mathcal E\subset\mathcal V$ be any nonempty set, and put
$S_*\defeq\sup_{s\in\Q_p}\dimH\widetilde P_s(\mathcal E)$.
For every $r_0\in\Q_p$,
\begin{equation}\label{eq:block-total-incidence}
\begin{aligned}
 S_*+1
 &\leq\dimH\widetilde{\mathcal I}_{\Q_p\setminus\{r_0\}}(\mathcal E)\\
 &=\dimH\widetilde{\mathcal I}(\mathcal E)
 \leq\min\{\dimH\mathcal E+1,N+1\}.
\end{aligned}
\end{equation}
In particular,
\begin{equation}\label{eq:block-incidence-sandwich}
 \dimH\widetilde P_0(\mathcal E)+1
 \leq\dimH\widetilde{\mathcal I}^{\times}(\mathcal E)
 \leq\min\{\dimH\mathcal E+1,N+1\}.
\end{equation}
\end{proposition}
\begin{proof}
We first write the block family as one vector-valued polynomial
family. Set $h\defeq\max_i h_i\geq1$, and put $b_{i,j}=0$
for $h_i<j\leq h$. For each power, combine its coefficients
into the vector
\begin{equation*}
 c_j\defeq(b_{1,j},\ldots,b_{k,j})\in\Q_p^N,
 \qquad 0\leq j\leq h,
\end{equation*}
and let $\mathcal C\subset(\Q_p^N)^{h+1}$ be their joint coefficient
set. Then
\begin{equation*}
 \widetilde P_r(\mathcal E)
   =\left\{\sum_{j=0}^h r^jc_j:c\in\mathcal C\right\}.
\end{equation*}
The map from $\mathcal E$ to $\mathcal C$ is linear and
Lipschitz. 

Fix $r_0\in\Q_p$. For $c\in\mathcal C$, define
\begin{equation*}
 c_j^{(r_0)}\defeq
       \sum_{j'=j}^h\binom{j'}j r_0^{j'-j}c_{j'},
 \qquad 0\leq j\leq h,
\end{equation*}
and let $\mathcal C_{r_0}$ be the resulting joint coefficient set.
For the unweighted family
$Q_t(w_0,\ldots,w_h)\defeq\sum_{j=0}^h t^jw_j$, the binomial
identity gives
\begin{equation*}
 Q_t(\mathcal C_{r_0})=\widetilde P_{r_0+t}(\mathcal E),
 \qquad Q_0(\mathcal C_{r_0})=\widetilde P_{r_0}(\mathcal E).
\end{equation*}
Moreover, $\dimH\mathcal C_{r_0}\leq\dimH\mathcal E$ because the
coefficient transformation is linear and Lipschitz.

Apply Proposition~\ref{prop:polynomial-total-incidence} in target
dimension $N$ to $Q_t$ and $\mathcal C_{r_0}$. Translating the
parameter by $r_0$ identifies its nonzero-parameter incidence set
isometrically with
$\widetilde{\mathcal I}_{\Q_p\setminus\{r_0\}}(\mathcal E)$.
Consequently,
\begin{equation}\label{eq:block-reference-incidence-step}
\begin{aligned}
 \dimH\widetilde P_{r_0}(\mathcal E)+1
 &\leq\dimH\widetilde{\mathcal I}_{\Q_p\setminus\{r_0\}}(\mathcal E)\\
 &\leq\min\{\dimH\mathcal E+1,N+1\}.
\end{aligned}
\end{equation}
It remains to replace the reference dimension by its supremum.
Observe that the omitted slice
$\{r_0\}\times\widetilde P_{r_0}(\mathcal E)$ has strictly
smaller dimension than the punctured incidence set. Adjoining this
slice therefore leaves the dimension unchanged. Since $r_0$ was
arbitrary, \eqref{eq:block-reference-incidence-step} gives a lower
bound for the full incidence dimension at every reference.
Taking the supremum proves \eqref{eq:block-total-incidence}.
Then \eqref{eq:block-incidence-sandwich} follows immediately.
\end{proof}

If some evaluation preserves the dimension of $\mathcal E$,
then the lower and upper bounds agree. In that case,
$\dimH\widetilde{\mathcal I}(\mathcal E)=\dimH\mathcal E+1$. We next estimate the dimensions
of individual images. 

\subsection*{Individual evaluation dimensions}
For the incidence bound, it was enough to pad every block to a
common degree. Applying the slice theorem to this regrouping would only
give a denominator of $Nh$ for the exceptional set, where $h=\max_i h_i$. To obtain the 
smaller denominator $\Lambda\leq Nh$, we need to lift the blocks
separately. 

\begin{theorem}[Block-valued polynomial projections]
\label{thm:block-valued-projections}
For the family \eqref{eq:block-evaluation-family}, let
$\mathcal E\subset\mathcal V$ be nonempty and analytic, and put
$S\defeq\dimH\widetilde P_0(\mathcal E)$. Then
\begin{equation}\label{eq:block-projection-ae}
 \dimH\widetilde P_r(\mathcal E)\geq S
 \quad\text{for Haar-almost every }r\in\Q_p.
\end{equation}
For every $0\leq\sigma<S$,
\begin{equation}\label{eq:block-projection-exceptional}
 \dimH\{r\in\Q_p:\dimH\widetilde P_r(\mathcal E)\leq\sigma\}
 \leq1-\frac{S-\sigma}{\Lambda},
\end{equation}
where $\Lambda$ is defined in \eqref{eq:block-denominator}.
When all $d_i=0$, the images are independent of the parameter and
the exceptional sets are empty.
\end{theorem}
\begin{proof}
Recall that $\mathcal A$ is nonempty and analytic. We shall
construct an analytic line family and show that its slope set
and each nonzero-parameter slice have the same additional
$L$ dimensions.

For each $b\in\mathcal A$, choose free vectors
$c_{i,1},\ldots,c_{i,h_i-1}\in\Q_p^{n_i}$, and set
$c_{i,0}\defeq b_{i,0}$ and $c_{i,h_i}\defeq0$. Define slope and
intercept vectors in $\Q_p^\Lambda$ by
\begin{equation*}
 u_{i,j}\defeq c_{i,j},\qquad
 w_{i,j}\defeq b_{i,j+1}-c_{i,j+1},
 \qquad 1\leq i\leq k,\quad0\leq j<h_i.
\end{equation*}
Let $Y$ consist of all resulting pairs $(u,w)$. It is a continuous
image of $\mathcal A\times\Q_p^L$, hence is nonempty and analytic,
and parametrizes lines $t\mapsto(t,w+tu)$ in $\Q_p^{\Lambda+1}$.
As in the single-output case, we can read off the slope set.
After a fixed coordinate permutation, it follows from Lemma~\ref{lem:products} that
\begin{equation}\label{eq:block-slope-product}
 D=\widetilde P_0(\mathcal E)\times\Q_p^L,
 \qquad\dimH D=S+L.
\end{equation}

We now compute the slices. For
$z=(z_{i,j})_{i,j}\in\Q_p^\Lambda$, define the linear map
$R_t:\Q_p^\Lambda\to\Q_p^N$ by
\begin{equation}\label{eq:block-slice-map}
 (R_tz)_i\defeq\sum_{j=0}^{h_i-1}t^{h_i-1-j}z_{i,j},
 \qquad1\leq i\leq k.
\end{equation}
Surjectivity follows because the coefficient of $z_{i,h_i-1}$
is one in every output block. Set
\begin{equation*}
 A_t\defeq\operatorname{diag}
       (t^{h_1}I_{n_1},\ldots,t^{h_k}I_{n_k}),
 \qquad H_t\defeq\{w+tu:(u,w)\in Y\},
\end{equation*}
where $I_{n_i}$ is the identity matrix on $\Q_p^{n_i}$.
Substituting the slope and intercept coordinates in each block,
we obtain
\begin{equation}\label{eq:block-telescoping}
\begin{aligned}
 (R_t(w+tu))_i
 &=\sum_{j=0}^{h_i-1}t^{h_i-1-j}
                  (b_{i,j+1}-c_{i,j+1}+tc_{i,j})\\
 &=\sum_{j=0}^{h_i}t^{h_i-j}b_{i,j}
 =t^{h_i}f_{i,b}(1/t),\qquad t\neq0.
\end{aligned}
\end{equation}
We claim that the resulting slice is the full inverse image,
as in \eqref{eq:projection-entire-fibers}:
\begin{equation}\label{eq:block-entire-fibers}
 H_t=R_t^{-1}\bigl(A_t\widetilde P_{1/t}(\mathcal E)\bigr),
 \qquad t\neq0.
\end{equation}
The inclusion from left to right follows from
\eqref{eq:block-telescoping}. For the other inclusion, choose
$b\in\mathcal A$ such that $R_tz=A_tf_b(1/t)$. Set
$c_{i,0}=b_{i,0}$, and define successively
\begin{equation*}
 c_{i,j+1}\defeq b_{i,j+1}+tc_{i,j}-z_{i,j},
 \qquad0\leq j<h_i.
\end{equation*}
The terminal value is
\begin{equation*}
 c_{i,h_i}
 =\sum_{j=0}^{h_i}t^{h_i-j}b_{i,j}
       -\sum_{j=0}^{h_i-1}t^{h_i-1-j}z_{i,j}
 =t^{h_i}f_{i,b}(1/t)-(R_tz)_i=0.
\end{equation*}
The intermediate vectors are admissible free variables, and their block equations
give $z=w+tu$. This proves the claim. The calculation also holds
when $h_i=1$, in which case there are no free variables in that
block.

To compute the dimension of this inverse image, observe that
for fixed $t$ the map
\begin{equation*}
 J_tz\defeq\bigl(R_tz,(z_{i,j})_{1\leq i\leq k,\,0\leq j<h_i-1}\bigr)
 \in\Q_p^N\times\Q_p^L
\end{equation*}
is an invertible linear coordinate change. Its inverse retains the
listed free coordinates and recovers each final block from
\begin{equation*}
 z_{i,h_i-1}
 =y_i-\sum_{j=0}^{h_i-2}t^{h_i-1-j}z_{i,j}.
\end{equation*}
Here, we take empty sums to be zero. For $t\neq0$, $A_t$ is also invertible.
For each fixed $t$, these invertible linear maps are bi-Lipschitz.
By \eqref{eq:block-entire-fibers} and Lemma~\ref{lem:products},
it follows that
\begin{equation}\label{eq:block-fiber-dimension}
 \dimH H_t=L+\dimH\widetilde P_{1/t}(\mathcal E),
 \qquad t\neq0.
\end{equation}

We now apply Theorem~\ref{thm:quantitative-line-slicing} with
$m=\Lambda+1$. The slope dimension satisfies
$S+L\leq N+L=\Lambda$.
Equations \eqref{eq:block-slope-product} and
\eqref{eq:block-fiber-dimension} give, for almost every $t\neq0$,
\begin{equation*}
 L+\dimH\widetilde P_{1/t}(\mathcal E)
 =\dimH H_t\geq S+L.
\end{equation*}
Subtracting $L$ gives the desired lower bound. For
$0\leq\sigma<S$, the exceptional-set assertion and the same
dimension identities give
\begin{align*}
 &\dimH\{t\neq0:
            \dimH\widetilde P_{1/t}(\mathcal E)\leq\sigma\}\notag\\
 &\qquad\leq1-\frac{(S+L)-(\sigma+L)}{\Lambda}
 =1-\frac{S-\sigma}{\Lambda}.
\end{align*}
Inversion is a similarity on each valuation annulus, so it preserves
Hausdorff dimension and Haar-null sets on $\Q_p^\times$, as proved
in the argument for Theorem~\ref{thm:polynomial-projections}.
We therefore obtain both conclusions for $r\neq0$.
At $r=0$, the image has dimension $S$, so this parameter is
not exceptional when $\sigma<S$. This proves
\eqref{eq:block-projection-ae} and
\eqref{eq:block-projection-exceptional}, completing the proof.
\end{proof}

\subsection*{Reference parameters and normalization}
We again consider replacing $r=0$ by an arbitrary reference parameter.
We translate
each block separately; this preserves $d_i$, $h_i$, and hence
$\Lambda$. We then choose a sequence of reference dimensions
approaching the supremum, as in
Corollary~\ref{cor:maximal-projection-dimension}.

\begin{corollary}[Reference parameters and maximal block-image dimension]
\label{cor:block-reference-maximal}
Under the hypotheses of Theorem~\ref{thm:block-valued-projections},
fix $r_0\in\Q_p$ and set
$S_{r_0}\defeq\dimH\widetilde P_{r_0}(\mathcal E)$. Then
\begin{equation*}
 \dimH\widetilde P_r(\mathcal E)\geq S_{r_0}
 \quad\text{for Haar-almost every }r\in\Q_p,
\end{equation*}
and, for every $0\leq\sigma<S_{r_0}$,
\begin{equation}\label{eq:block-reference-exceptional}
 \dimH\{r\in\Q_p:\dimH\widetilde P_r(\mathcal E)\leq\sigma\}
 \leq1-\frac{S_{r_0}-\sigma}{\Lambda}.
\end{equation}
More strongly, with
$S_*\defeq\sup_{s\in\Q_p}\dimH\widetilde P_s(\mathcal E)$, one has
\begin{equation}\label{eq:block-maximal-ae}
 \dimH\widetilde P_r(\mathcal E)=S_*
 \quad\text{for Haar-almost every }r\in\Q_p,
\end{equation}
and, for every $0\leq\sigma<S_*$,
\begin{equation}\label{eq:block-maximal-exceptional}
 \dimH\{r\in\Q_p:\dimH\widetilde P_r(\mathcal E)\leq\sigma\}
 \leq1-\frac{S_*-\sigma}{\Lambda}.
\end{equation}
If $\dimH\widetilde P_{r_0}(\mathcal E)=\dimH\mathcal E$ for some
$r_0$, then $\dimH\widetilde P_r(\mathcal E)=\dimH\mathcal E$
for Haar-almost every $r$.
\end{corollary}
\begin{proof}
We first prove the fixed-reference statements. Use the absorbed
coefficients from \eqref{eq:block-absorbed-coefficients}, indexed by
the original degrees $0\leq j\leq d_i$. The auxiliary zero in a
constant block is not needed for the following expansion:
\begin{equation*}
\begin{aligned}
 f_{i,b}(r_0+t)&=\sum_{j=0}^{d_i}\widehat b_{i,j}t^j,\\
 \widehat b_{i,j}&\defeq\sum_{j'=j}^{d_i}
                    \binom{j'}j r_0^{j'-j}b_{i,j'}.
\end{aligned}
\end{equation*}
The joint transformation $b\mapsto\widehat b$ is continuous and linear.
Its leading coefficient in block $i$ is still $b_{i,d_i}$, and
its constant-coefficient image is exactly
$\widetilde P_{r_0}(\mathcal E)$.
Apply Theorem~\ref{thm:block-valued-projections} to this transformed
analytic coefficient set, with $d_i+1$ coefficient vectors and all
fixed weights equal to one in block $i$. Its lifting parameter is
again $\sum_i n_i h_i=\Lambda$. The resulting evaluation
images are exactly $\widetilde P_{r_0+t}(\mathcal E)$.
Translation preserves Hausdorff dimension and Haar-null sets,
so the fixed-reference assertions follow.

We now prove the maximal-dimension assertions. If $S_*=0$,
all image dimensions are zero. Otherwise, choose a sequence
$r_j\in\Q_p$ with
$S_j\defeq\dimH\widetilde P_{r_j}(\mathcal E)\to S_*$.
For each $j$, the preceding assertion gives
$\dimH\widetilde P_r(\mathcal E)\geq S_j$ outside a Haar-null set.
Outside the union of these countably many null sets,
\begin{equation*}
 \dimH\widetilde P_r(\mathcal E)\geq\sup_j S_j=S_*.
\end{equation*}
The reverse inequality holds for every $r$ by definition of $S_*$,
proving \eqref{eq:block-maximal-ae}. For a fixed
$0\leq\sigma<S_*$, apply \eqref{eq:block-reference-exceptional}
with $r_0=r_j$ whenever $S_j>\sigma$ and then let $j\to\infty$.
The exceptional set being estimated is independent of $j$, which
gives \eqref{eq:block-maximal-exceptional}.
Finally, every fixed $\widetilde P_r$ is linear and Lipschitz, so
\begin{equation*}
 \dimH\widetilde P_{r_0}(\mathcal E)
 \leq S_*\leq\dimH\mathcal E.
\end{equation*}
Together with \eqref{eq:block-maximal-ae}, this proves the
last assertion and completes the proof.
\end{proof}

\medskip
\section{Anisotropic covering numbers and reference parameters}
\label{sec:anisotropic-covering}

We continue with the fixed block family
\eqref{eq:block-evaluation-family} and the lifting parameters
\eqref{eq:block-lifting-dimensions}. Unlike the individual-image
dimension theorem, the covering comparison requires no regularity
of the coefficient set. Its hypothesis is boundedness of the
absorbed coefficients.
The different block scales will be reduced to one residue modulus.
Since this changes the coefficient set with the scale, we need
an estimate that is uniform in that set.

We first give the covering definition and state the quantitative
result. We then prove a uniform finite-ring comparison using
Corollary~\ref{cor:dhar-directions}. After normalizing the block
scales, that comparison gives the quantitative theorem.
Borel--Cantelli will then give one exceptional null set for all
reference parameters.

\subsection*{Covering numbers and the quantitative statement}
We use the following definition of anisotropic covering number.
A Euclidean rectangular-partition version is given in
\cite[Section~2.1]{BenardHe}. Here the rectangles are lattice
cosets. Consequently, counting the cosets that meet a set gives
exactly its minimal covering number.

\begin{definition}[Anisotropic covering numbers]
\label{def:block-rectangles}
Fix $\boldsymbol\alpha=(\alpha_1,\ldots,\alpha_k)$ with
$\alpha_i\geq0$. For $0<\delta<1$, put
\begin{equation}\label{eq:block-rectangle-lattice}
 m_i(\delta)\defeq
       \left\lceil\alpha_i\log_p(1/\delta)\right\rceil,
 \qquad
 \mathcal Q_\delta^{\boldsymbol\alpha}
       \defeq\prod_{i=1}^k p^{m_i(\delta)}\Z_p^{n_i}.
\end{equation}
A $\delta$-rectangle is a translate of
$\mathcal Q_\delta^{\boldsymbol\alpha}$. Its $i$th block is a ball
of radius $\delta^{\alpha_i}$. For $A\subset\prod_i\Q_p^{n_i}$, let
$N_\delta^{\boldsymbol\alpha}(A)$ denote the number of these cosets
that meet $A$, equivalently this is the least number of such rectangles
covering $A$. This number is finite for bounded $A$. It is $0$ for
the empty set and $\infty$ for an unbounded set. At a discrete scale
$\delta=p^{-\ell}$, one has $m_i(\delta)=\lceil\alpha_i\ell\rceil$.
\end{definition}

The connection with diagonal expansion is given by an exact
rescaling identity. Set
\begin{equation*}
 \mathsf A_\delta
 \defeq \operatorname{diag}
       \bigl(p^{-m_i(\delta)}I_{n_i}\bigr)_{i=1}^k,
\end{equation*}
and let $N_1^{\mathrm{iso}}$ denote the covering number by
translates of $\Z_p^N$. Since
$\mathsf A_\delta\mathcal Q_\delta^{\boldsymbol\alpha}=\Z_p^N$,
we have
\begin{equation*}
 N_\delta^{\boldsymbol\alpha}(A)
 =N_1^{\mathrm{iso}}(\mathsf A_\delta A).
\end{equation*}
Thus the anisotropic covering number is the unit-scale covering
number after diagonal rescaling. The proof below uses a different
normalization, contracting the blocks to a common residue modulus
so that the coefficient bounds remain uniform.

Let $\mathcal E\subset\mathcal V$ be nonempty, and let
$\mathcal A$ be its joint absorbed coefficient set from
\eqref{eq:block-absorbed-coefficients}. Recall that
\begin{equation*}
 \widetilde P_r(\mathcal E)
     =\left\{\left(\sum_{j=0}^{h_i}b_{i,j}r^j\right)_i:
                                                    b\in\mathcal A\right\}.
\end{equation*}
Only boundedness of the absorbed coefficient set $\mathcal A$
is required below. Boundedness of $\mathcal E$ is sufficient,
but need not be necessary when some fixed scalar weights vanish.

For a compact parameter ball $B=a+c\Z_p$, $c\neq0$, we
write
\begin{equation*}
 \mu_B(A)\defeq\frac{\mathcal L^1(A\cap B)}{\mathcal L^1(B)}
\end{equation*}
for its normalized Haar measure.

\begin{theorem}[Quantitative anisotropic covering comparison]
\label{thm:block-anisotropic-covering}
For the family \eqref{eq:block-evaluation-family}, fix
$\boldsymbol\alpha\in[0,\infty)^k$, and let
$\mathcal E\subset\mathcal V$ be nonempty with bounded joint absorbed
coefficient set $\mathcal A$. For a compact parameter ball
$B\subset\Q_p$ of positive radius, put
\begin{equation*}
 M_{\delta,B}\defeq
          \max_{s\in B}N_\delta^{\boldsymbol\alpha}
                              (\widetilde P_s(\mathcal E)).
\end{equation*}
This maximum exists and is finite. There is a constant $C$ such that,
for every $\ell\geq1$, $\delta=p^{-\ell}$, and $0<\eta<1$,
\begin{equation}\label{eq:block-anisotropic-maximal-tail}
 \mu_B\left\{r\in B:
       N_\delta^{\boldsymbol\alpha}(\widetilde P_r(\mathcal E))
                      <\eta M_{\delta,B}\right\}
 \leq C(1+\ell)^{(\Lambda+1)/\Lambda}\eta^{1/\Lambda}.
\end{equation}
In particular, the same bound holds with $M_{\delta,B}$ replaced by
$N_\delta^{\boldsymbol\alpha}(\widetilde P_{r_0}(\mathcal E))$,
uniformly for all $r_0\in B$. The constant $C$ depends only on
$p,(n_i),(d_i),\boldsymbol\alpha,B$, and the fixed bound for
$\mathcal A$, not on $\ell,\eta,r_0$, or the particular coefficient
set within that bound. 
\end{theorem}

To prove the theorem, we first establish a finite-ring estimate
with a constant independent of the coefficient set and the
reference parameter. We then normalize the anisotropic scales
and apply this estimate. We give the finite-ring argument next.

\subsection*{A uniform finite-ring comparison}
The following lemma is stated independently of the weighted
$\Q_p$-family, for arbitrary finite-ring coefficients and degree
bounds $h_i\geq1$. These bounds need not be attained. This is the
form needed after normalization and reduction modulo a power of $p$.

A finite-field antecedent is the algebraic-curve incidence estimate of
Ellenberg, Oberlin, and Tao \cite[Corollary~1.10]{EOT}. A corresponding
evaluation comparison follows by applying their estimate to polynomial
graphs through the distinct points of a reference image.

\begin{lemma}[Uniform finite-ring evaluation comparison]
\label{lem:block-finite-ring-comparison}
Let $\nu\geq1$, $\mathcal R\defeq\mathcal R_\nu$, and $q\defeq p^\nu$.
Fix integers $k\geq1$, output dimensions $n_i\geq1$, and degree
bounds $h_i\geq1$. Set
$\Lambda=\sum_i n_i h_i$ and $L=\sum_i n_i(h_i-1)$.
For a nonempty set
\begin{equation*}
 F\subset\prod_{i=1}^k(\mathcal R^{n_i})^{h_i+1},
 \qquad
 \mathsf P_r(b)\defeq
       \left(\sum_{j=0}^{h_i}b_{i,j}r^j\right)_{i=1}^k,
\end{equation*}
one has, for every $r_0\in\mathcal R$ and $0<\eta<1$,
\begin{equation}\label{eq:block-finite-ring-comparison}
 \frac1q\#\{r\in\mathcal R:
          \#\mathsf P_r(F)<\eta\,\#\mathsf P_{r_0}(F)\}
 \leq C_{p,\Lambda}(1+\nu)^{(\Lambda+1)/\Lambda}
                   \eta^{1/\Lambda}.
\end{equation}
The same bound holds with $\#\mathsf P_{r_0}(F)$ replaced by
$\max_{u\in\mathcal R}\#\mathsf P_u(F)$. No nonvanishing
assumption is imposed on the coefficient vectors. The constant $C$ is independent of
$\nu,F,r_0$, and $\eta$.
\end{lemma}
\begin{proof}
We first prove the lemma for $r_0=0$ and unit parameters.
Define
\begin{equation*}
 K_0\defeq\#\mathsf P_0(F),\qquad
 T\defeq\{r\in\mathcal R^\times:
                         \#\mathsf P_r(F)<\eta K_0\},
 \qquad U\defeq T^{-1}.
\end{equation*}
If $T$ is empty, the unit-parameter bound is immediate.
Assume that $T$ is nonempty. For $t\in U$, define
\begin{equation*}
 A_t\defeq
    \operatorname{diag}(t^{h_i}I_{n_i})_{i=1}^k\,
                  \mathsf P_{t^{-1}}(F).
\end{equation*}
Every $t\in U$ is a unit, so $\#A_t=\#\mathsf P_{t^{-1}}(F)<\eta K_0$.
Use the finite-ring version of \eqref{eq:block-slice-map}, namely
\begin{equation*}
 (R_tz)_i\defeq
       \sum_{j=0}^{h_i-1}t^{h_i-1-j}z_{i,j},
 \qquad z\in\mathcal R^\Lambda,
\end{equation*}
and form the incidence set
\begin{equation*}
 X\defeq\{(t,z)\in U\times\mathcal R^\Lambda:R_tz\in A_t\}.
\end{equation*}
We estimate $\#X$ in two ways. First, the coefficient of
$z_{i,h_i-1}$ in each block is $1$. Hence $R_t$ is surjective
and every fiber contains exactly $q^L$ points. It follows that
\begin{equation}\label{eq:block-finite-incidence-upper}
 \#X=q^L\sum_{t\in U}\#A_t
       \leq\eta K_0q^L\#T.
\end{equation}

For the lower bound, we construct lines meeting $X$ in many
points. For each $a=(b_{i,0})_i\in\mathsf P_0(F)$, exists $b\in F$ such that $P_0(b)=a$.
For every choice of auxiliary vectors
$c_{i,j}\in\mathcal R^{n_i}$, $1\leq j<h_i$, set
$c_{i,0}\defeq b_{i,0}$ and $c_{i,h_i}\defeq0$, and consider the line
\begin{equation*}
 \mathcal L_{b,c}(t)\defeq
 \left(t,
       (b_{i,j+1}-c_{i,j+1}+tc_{i,j})_{i,\,0\leq j<h_i}\right),
 \qquad t\in\mathcal R.
\end{equation*}
Write $z(t)$ for the second component. The same calculation
as in \eqref{eq:block-telescoping}, now over $\mathcal R$, gives
\begin{align*}
 (R_tz(t))_i
 &=\sum_{j=0}^{h_i-1}t^{h_i-1-j}
             (b_{i,j+1}-c_{i,j+1}+tc_{i,j})\notag\\
 &=\sum_{j=0}^{h_i}t^{h_i-j}b_{i,j}
  =t^{h_i}\mathsf P^i_{t^{-1}}(b),\qquad t\in U.
\end{align*}
Hence each such line meets $X$ in at least $\#T$ points. Its direction
has the primitive representative
\[
 (1,(c_{i,j})_{i,\,0\leq j<h_i}).
\]
The constant slope blocks range over $\mathsf P_0(F)$ and
the free slope blocks range over $\mathcal R^L$. Since the
first coordinate is $1$, distinct such tuples represent distinct
projective directions. Thus the number of directions is exactly
$K_0q^L$.

Apply Corollary~\ref{cor:dhar-directions}, which follows from
Dhar's set estimate \cite[Theorem~1.8]{DharSet}, in ambient dimension
$\Lambda+1$. It gives
\begin{equation}\label{eq:block-finite-incidence-lower}
 \#X\geq c_{p,\Lambda}(1+\nu)^{-(\Lambda+1)}
        \frac{K_0q^L(\#T)^{\Lambda+1}}{q^\Lambda}.
\end{equation}
Combining \eqref{eq:block-finite-incidence-upper} and
\eqref{eq:block-finite-incidence-lower}, and cancelling
$K_0q^L\#T>0$, yields
\begin{equation*}
 \left(\frac{\#T}{q}\right)^\Lambda
       \leq C_{p,\Lambda}(1+\nu)^{\Lambda+1}\eta.
\end{equation*}
Taking the $\Lambda$th root proves the unit-parameter estimate.
In particular, the constant is independent of $F$.

We next include the nonunit parameters. For $0\leq a<\nu$,
define
\begin{equation*}
 F^{(a)}\defeq
       \{(p^{aj}b_{i,j})_{i,j}:b\in F\}.
\end{equation*}
Then $\mathsf P_u(F^{(a)})=\mathsf P_{p^a u}(F)$ and
$\mathsf P_0(F^{(a)})=\mathsf P_0(F)$. Let
\begin{equation*}
 T_a\defeq\{u\in\mathcal R^\times:
                         \#\mathsf P_u(F^{(a)})<\eta K_0\}.
\end{equation*}
The map $u\mapsto p^a u$ sends $\mathcal R^\times$ onto
$p^a\mathcal R^\times$, and each fiber has size $p^a$.
Consequently, the number of exceptional parameters in this shell
is $p^{-a}\#T_a$. If we let
\begin{equation*}
 \mathcal B\defeq
 \{r\in\mathcal R:\#\mathsf P_r(F)<\eta K_0\}.
\end{equation*}
Then
\begin{equation*}
    \begin{aligned}
 \frac{\#\mathcal B}{q}
 &=\sum_{a=0}^{\nu-1}p^{-a}\frac{\#T_a}{q}\\
 &\leq C_{p,\Lambda}(1+\nu)^{(\Lambda+1)/\Lambda}
       \eta^{1/\Lambda}\sum_{a=0}^{\nu-1}p^{-a}\\
 &\leq \frac{C_{p,\Lambda}}{1-p^{-1}}
       (1+\nu)^{(\Lambda+1)/\Lambda}\eta^{1/\Lambda}.
\end{aligned}
\end{equation*}
This proves \eqref{eq:block-finite-ring-comparison} for $r_0=0$.

It remains to treat an arbitrary reference $r_0\in\mathcal R$.
For this, replace the coefficients by
\begin{equation*}
 \widehat b_{i,j}\defeq
       \sum_{j'=j}^{h_i}\binom{j'}j r_0^{j'-j}b_{i,j'}.
\end{equation*}
The binomial identity gives
$\sum_j\widehat b_{i,j}t^j=\sum_jb_{i,j}(r_0+t)^j$ in $\mathcal R$.
The transformed constant-coefficient image is $\mathsf P_{r_0}(F)$,
and translating the parameter preserves counting measure. Applying
the already proved case gives
\eqref{eq:block-finite-ring-comparison} with the same constant.
Finally, the parameter set is finite, so a maximizing reference
$r_0\in\mathcal R$ exists. Apply the estimate at that reference
to obtain the maximal assertion. This completes the proof.
\end{proof}

\subsection*{Normalization and simultaneous reference parameters}
We now prove the quantitative covering theorem. We identify covering numbers with residue-class counts.
After a fixed initial normalization, all scale-dependent changes
will be contractions. We can then apply the preceding lemma
uniformly to the resulting coefficient sets.

\begin{proof}[Proof of Theorem~\ref{thm:block-anisotropic-covering}]
Write $B=a+c\Z_p$ with $c\neq0$. We first put the parameter
in $\Z_p$ by writing $r=a+cu$. For $b\in\mathcal A$, the
new coefficients are
\begin{equation*}
 \beta_{i,j}(b)\defeq
       c^j\sum_{j'=j}^{h_i}\binom{j'}j a^{j'-j}b_{i,j'},
 \qquad
 \sum_{j=0}^{h_i}b_{i,j}(a+cu)^j
       =\sum_{j=0}^{h_i}\beta_{i,j}(b)u^j.
\end{equation*}
Since $\mathcal A$ is bounded, the new joint coefficient set
is bounded in terms of $B$, the block degrees, and the bound
for $\mathcal A$. We may therefore choose a fixed integer
$\kappa\geq0$ such that
$p^\kappa\beta_{i,j}(b)\in\Z_p^{n_i}$ for every $b \in \mathcal{A}$ and indices $i,j$.

Fix $\delta=p^{-\ell}$ and put
\begin{equation*}
 m_i\defeq\lceil\alpha_i\ell\rceil,\qquad
 m_*\defeq\max\{1,m_1,\ldots,m_k\},\qquad
 \nu\defeq\kappa+m_*.
\end{equation*}
Multiply the $i$th output block, and all coefficients in that block, by
$p^{\kappa+m_*-m_i}$. The rescaled coefficients
\begin{equation*}
 \widehat b_{i,j}(b)\defeq
           p^{\kappa+m_*-m_i}\beta_{i,j}(b)
\end{equation*}
are integral. Observe that $m_*-m_i\geq0$.
After multiplication by the fixed initial factor, every further
multiplication is therefore a contraction. In particular, the
coefficient bound remains uniform as $\ell$ increases.

Let $F\subset\prod_i(\mathcal R_\nu^{n_i})^{h_i+1}$ be the
reduction of the rescaled coefficient set. We now check that the
covering rectangles correspond exactly to residue classes.
For two output vectors $y,y'$, we have
\begin{equation*}
 y-y'\in\prod_i p^{m_i}\Z_p^{n_i}
 \quad\Longleftrightarrow\quad
 \bigl(p^{\kappa+m_*-m_i}(y_i-y_i')\bigr)_i
                               \in p^\nu\Z_p^N.
\end{equation*}
Consequently, with the evaluation maps from
Lemma~\ref{lem:block-finite-ring-comparison},
\begin{equation}\label{eq:block-covering-residue-identity}
 N_\delta^{\boldsymbol\alpha}
       (\widetilde P_{a+cu}(\mathcal E))
       =\#\mathsf P_{\bar u}(F),\qquad u\in\Z_p,
\end{equation}
where $\bar u$ is reduction modulo $p^\nu$.
Thus, at each fixed scale, the covering number depends only on
the parameter residue class. Its maximum on $B$ is
$\max_{v\in\mathcal R_\nu}\#\mathsf P_v(F)$ which is finite.

The parametrization $u\mapsto a+cu$ pushes normalized Haar measure on
$\Z_p$ to $\mu_B$. Each residue class modulo $p^\nu$ has mass $p^{-\nu}$.
Applying the maximal version of
Lemma~\ref{lem:block-finite-ring-comparison} to
\eqref{eq:block-covering-residue-identity} therefore bounds the
left-hand side of \eqref{eq:block-anisotropic-maximal-tail} by
\[
 C_{p,\Lambda}(1+\nu)^{(\Lambda+1)/\Lambda}\eta^{1/\Lambda}.
\]
Since $\nu=\kappa+\max\{1,\lceil\alpha_1\ell\rceil,
\ldots,\lceil\alpha_k\ell\rceil\}=O(1+\ell)$, this proves
\eqref{eq:block-anisotropic-maximal-tail} with the stated dependencies.
For every $r_0\in B$, the reference covering number is at most
$M_{\delta,B}$. Its exceptional set is therefore contained in
the set just estimated. This proves the remaining assertion
and completes the proof.
\end{proof}

Taking $\eta=\delta^\varepsilon$ gives the explicit estimate
\begin{equation}\label{eq:block-anisotropic-power-tail}
 \mu_B\left\{r\in B:
       N_\delta^{\boldsymbol\alpha}(\widetilde P_r(\mathcal E))
                      <\delta^\varepsilon M_{\delta,B}\right\}
 \leq C(1+\ell)^{(\Lambda+1)/\Lambda}
                     p^{-\ell\varepsilon/\Lambda},
 \qquad\delta=p^{-\ell}.
\end{equation}
The right-hand side is summable in $\ell$, so Borel--Cantelli
applies. Since the maximum over references is already included
at each scale, we obtain an eventual comparison with all references
in a fixed compact ball. We now use a countable exhaustion to
obtain the following simultaneous statement.

\begin{corollary}[A common exceptional set for all reference parameters]
\label{cor:block-anisotropic-common-reference}
Under the hypotheses of Theorem~\ref{thm:block-anisotropic-covering},
there is a Haar-full set $G\subset\Q_p$ such that, for every $r\in G$,
every compact parameter ball $B\subset\Q_p$ of positive radius, and
every $\varepsilon>0$, there is
$\delta_0=\delta_0(r,B,\varepsilon)>0$ for which
\begin{equation}\label{eq:block-anisotropic-compact-uniform}
 N_\delta^{\boldsymbol\alpha}(\widetilde P_r(\mathcal E))
       \geq\delta^\varepsilon
          \max_{s\in B}N_\delta^{\boldsymbol\alpha}
                              (\widetilde P_s(\mathcal E)),
 \qquad 0<\delta<\delta_0.
\end{equation}
Here $\delta$ may be any real scale, and $r$ need not belong to $B$.
In particular, the same set $G$ works simultaneously for every
$r_0\in\Q_p$ and every $\varepsilon>0$ in the comparison
\begin{equation}\label{eq:block-anisotropic-arbitrary-reference}
 N_\delta^{\boldsymbol\alpha}(\widetilde P_r(\mathcal E))
       \geq\delta^\varepsilon
             N_\delta^{\boldsymbol\alpha}
                              (\widetilde P_{r_0}(\mathcal E))
 \quad\text{for all sufficiently small }\delta.
\end{equation}
The full-measure set may depend on the fixed coefficient set and
$\boldsymbol\alpha$, but not on the reference parameter or
$\varepsilon$. On each fixed compact reference ball, the scale threshold
is uniform in the reference parameter.
\end{corollary}
\begin{proof}
We first prove the conclusion at the discrete scales. For each
integer $A\geq0$, let $B_A\defeq p^{-A}\Z_p$. For fixed $A\geq0$ and $\varepsilon>0$, define the bad sets
\begin{equation*}
 \mathcal B_\ell\defeq
 \left\{r\in B_A:
 N_{p^{-\ell}}^{\boldsymbol\alpha}
       (\widetilde P_r(\mathcal E))
 <p^{-\ell\varepsilon}M_{p^{-\ell},B_A}\right\}.
\end{equation*}
By \eqref{eq:block-anisotropic-power-tail},
\begin{equation*}
 \sum_{\ell=1}^{\infty}\mu_{B_A}(\mathcal B_\ell)
 \leq C\sum_{\ell=1}^{\infty}
 (1+\ell)^{(\Lambda+1)/\Lambda}
 p^{-\ell\varepsilon/\Lambda}<\infty.
\end{equation*}
The Borel--Cantelli lemma therefore gives
$\mu_{B_A}(\limsup_{\ell\to\infty}\mathcal B_\ell)=0$.
Thus almost every $r\in B_A$ belongs to only finitely many
$\mathcal B_\ell$. Equivalently, for such $r$ there exists
$\ell_0=\ell_0(r,A,\varepsilon)$ such that
\begin{equation*}
 N_{p^{-\ell}}^{\boldsymbol\alpha}
       (\widetilde P_r(\mathcal E))
 \geq p^{-\ell\varepsilon}M_{p^{-\ell},B_A}
 \qquad\text{for every }\ell\geq\ell_0.
\end{equation*}
Remove the exceptional null sets for all
$A\geq0$ and all $\varepsilon=1/j$, $j\geq1$, and denote the resulting
Haar-full subset of $\Q_p$ by $G$. For $0<\delta<1$, a comparison
with exponent $1/j<\varepsilon$ implies the one with exponent
$\varepsilon$. Thus $G$ works for every positive exponent.

Given $r\in G$ and a compact parameter ball $B$, choose $A$ so large
that $\{r\}\cup B\subset B_A$. The maximal comparison on $B_A$
implies the one on $B$, with a scale threshold independent of the
reference point in $B$. This proves
\eqref{eq:block-anisotropic-compact-uniform} at discrete scales.
Observe that we have removed only countably many null sets.
The maximum over references was taken before applying
Borel--Cantelli, so no uncountable intersection of full-measure
sets is involved.

It remains to pass from the discrete scales to all real scales.
Set
\begin{equation*}
 C_{\boldsymbol\alpha}\defeq
             p^{\sum_i n_i\lceil\alpha_i\rceil}.
\end{equation*}
If $p^{-(\ell+1)}<\delta\leq p^{-\ell}$, the lattice inclusions
from \eqref{eq:block-rectangle-lattice} and their indices give,
for every bounded set $Y$,
\begin{equation*}
 N_{p^{-\ell}}^{\boldsymbol\alpha}(Y)
 \leq N_\delta^{\boldsymbol\alpha}(Y)
 \leq N_{p^{-(\ell+1)}}^{\boldsymbol\alpha}(Y)
 \leq C_{\boldsymbol\alpha}
                         N_{p^{-\ell}}^{\boldsymbol\alpha}(Y).
\end{equation*}
Apply the discrete-scale comparison with exponent $\varepsilon/2$.
It follows that
\begin{equation*}
 N_\delta^{\boldsymbol\alpha}(\widetilde P_r(\mathcal E))
       \geq C_{\boldsymbol\alpha}^{-1}p^{-\ell\varepsilon/2}
                         M_{\delta,B}
       \geq\delta^\varepsilon M_{\delta,B}
\end{equation*}
for all sufficiently large $\ell$.
For any fixed $r_0$, choose a compact ball $B$ containing it.
The maximal comparison on $B$ implies
\eqref{eq:block-anisotropic-arbitrary-reference}. This completes
the proof.
\end{proof}

We can now deduce Theorem~\ref{thm:intro-anisotropic-maximality}.
If $\mathcal E$ is bounded, its absorbed coefficient set
$\mathcal A$ is bounded. The preceding corollary therefore
applies and gives all the stated conclusions.

\begin{remark}[Compact reference ranges are necessary for uniform thresholds]
\label{rem:block-reference-threshold}
The exceptional null set is independent of the reference.
However, the scale threshold need not be uniform over all of
$\Q_p$. To see this, take
$P_s(a,b)=a+sb$ and $\mathcal E=\{(0,b):b\in\Z_p\}$.
The isotropic covering numbers satisfy
\begin{equation*}
 P_s(\mathcal E)=s\Z_p,
 \qquad N_{p^{-\ell}}(P_{p^{-a}}(\mathcal E))=p^{a+\ell}
                       \quad(a\geq0,\ \ell\geq1).
\end{equation*}
At each fixed scale the reference covering numbers are unbounded as
$a\to\infty$, whereas each fixed evaluation image has a finite
covering number. Thus the compact reference range in
\eqref{eq:block-anisotropic-compact-uniform} cannot generally be
replaced by all of $\Q_p$.
\end{remark}

\medskip
\bibliographystyle{amsalpha}

\end{document}